\documentclass[11pt,reqno]{amsart}
\usepackage{amsmath,amsfonts,amsthm,mathrsfs,amssymb,amscd,comment,enumerate,amsxtra,url,tikz-cd}
\input{xypic}
\xyoption{all}

\usepackage{xcolor}
\colorlet{mdtRed}{red!50!black}
\definecolor{dblue}{rgb}{0,0,.6}
\usepackage[colorlinks]{hyperref}
\hypersetup{linkcolor=blue,citecolor=dblue,filecolor=dullmagenta,urlcolor=mdtRed}

\numberwithin{equation}{subsection}
\newtheorem{theorem}[equation]{Theorem}
\newtheorem{corollary}[equation]{Corollary}
\newtheorem{lemma}[equation]{Lemma}
\newtheorem{proposition}[equation]{Proposition}
\newtheorem{definition}[equation]{Definition}

\newtheorem*{theorem*}{Theorem}
\newtheorem*{corollary*}{Corollary}

\theoremstyle{remark}
\newtheorem{remark}[equation]{Remark}

\newcommand{\Z}{\mathbb{Z}}
\newcommand{\C}{\mathbb{C}}

\DeclareMathOperator{\et}{\textnormal{\'et}}

\newcommand{\mf}[1]{\mathfrak{#1}}
\newcommand{\ms}[1]{\mathscr{#1}}
\newcommand{\mb}[1]{\mathbb{#1}}
\newcommand{\mc}[1]{\mathcal{#1}}

\begin{document}
	
\title[Recent results on Punctual Quot Schemes]{Some recent results on the punctual Quot schemes} 

	\author[I. Biswas]{Indranil Biswas} 
	
	\address{Department of Mathematics, Shiv Nadar University, NH91, Tehsil Dadri, Greater Noida, Uttar
		Pradesh 201314, India} 
	
	\email{indranil.biswas@snu.edu.in, indranil29@gmail.com} 
	
	\author[C. Gangopadhyay]{Chandranandan Gangopadhyay} 
	
	\address{Department of Mathematics, Indian Institute of Science Education and Research, Pune, 411008, Maharashtra, India.}
	
	\email{chandranandan@iiserpune.ac.in}
	
	\author[R. Sebastian]{Ronnie Sebastian} 
	
	\address{Department of Mathematics, Indian Institute of Technology Bombay, Powai, Mumbai 
		400076, Maharashtra, India}
	
	\email{ronnie@math.iitb.ac.in} 
	
	
	\date{\today}
	
	\begin{abstract}
			Let $C$ be a smooth projective curve defined over the field of complex numbers. 
			Let $E$ be a vector bundle on $C$, and fix an integer $d\,\geqslant  \, 1$. 
			Let $\mc Q\,:=\,{\rm Quot}(E,d)$ be the Quot Scheme which parameterizes all
			torsion quotients of $E$ of degree $d$. In this article, 
			we survey some recent results on various invariants of $\mc Q$.
	\end{abstract}
	
	\maketitle

\tableofcontents

	\maketitle
	
\section{Introduction}

Let $C$ be a smooth projective curve over the field of complex numbers. 
	Let $E$ be a vector bundle on $C$, and fix an integer $d\,\geqslant   \,1$.
	Let $\mc Q\,:=\,{\rm Quot}(E,d)$ be the Quot Scheme which parameterizes all
	torsion quotients of $E$ of degree $d$. Then $\mc Q$ is a smooth 
	projective variety which admits various moduli theoretic interpretations. 
	For example, in \cite{BGL} Bifet, Ghione and Letizia interpreted 
	this Quot scheme as the space of higher rank effective divisors and 
went on to construct a higher dimensional analogue of the Abel-Jacobi map. 
	They used this map to compute the Betti numbers of the moduli space of vector bundles. 
	This approach has been used by Hoskins and Lehalleur \cite{HP} to 
	study the Voevodsky motives of moduli space of bundles over curves 
	as well. The Quot scheme $\mc Q$ can also be thought as the moduli 
	space of vortices of a certain numerical type \cite{Biswas-Romao}. 

Recently, there has been a lot of research in understanding various 
aspects of the geometry of $\mc Q$, and many invariants associated 
to $\mc Q$ has been computed. 
See for example, \cite{BDH-aut}, \cite{G19} \cite{BGS}, for automorphism 
groups of $\mc Q$, \cite{Bis-Dh-Hu} for the Brauer group of $\mc Q$, \cite{GS} 
for the $S$-fundamental group, \cite{GS-nef} for the nef cone of $\mc Q$. 
There has been significant efforts towards understanding the cohomology of 
various natural bundles on $\mc Q$. In \cite{BGS} the cohomology of 
tangent bundle was computed. In \cite{OS} the Euler characteristic 
of certain tautological bundles was computed, and a conjectural formula 
was given by for their cohomology. In \cite{Krug}, the author proves this conjecture 
in various cases. In \cite{Marian-Negut}, the authors give a representation-theoretic interpretation of
the cohomology (with $\mathbb Q$ coefficients) of $\mc Q$.  

In this article, we survey some of the recent results on various invariants associated to $\mc Q$. Broadly, we discuss the
computations of three invariants of $\mc Q$: The $S$-fundamental group scheme, the nef cone and the cohomology of tangent bundle of 
$\mc Q$. One of the crucial ingredients in these computations is a study of a map called Hilbert-Chow map, which is a map from $\mc Q$ 
to $C^{(d)}$, the $d$-th symmetric product of the curve $C$; see Section \ref{Hilbert-Chow} for the
construction of this map. More specifically, we 
study the fibers of this map. The fibers are not smooth but there exists certain natural resolution of this spaces. We use these 
resolutions to reduce various computation on $\mc Q$ to computations on $C^{(d)}$ (see Lemma \ref{the map g_d} and Corollary 
\ref{cor-cohomology over Q_D and S_D}).
	
	\section{$S$-fundamental group scheme of ${\rm Quot}(E,d)$}

	Let $X$ be a connected, reduced and complete scheme over a perfect field $k$,
	and let $x \,\in\, X$ be a $k$-rational point.
	In \cite{No1}, Nori introduced a $k$-group scheme $\pi^N(X,\,x)$ associated to essentially 
	finite vector bundles on $X$. In \cite{No2}, Nori extends the definition of $\pi^N(X,\,x)$ 
	to connected and reduced $k$-schemes. 
	In \cite{BPS}, Biswas, Parameswaran and Subramanian defined the notion of {\it $S$-fundamental group scheme} 
	$\pi^S(X,\, x)$ for $X$ a smooth projective curve over any algebraically closed field $k$. 
	This is generalized to higher dimensional connected smooth projective $k$-schemes 
	and studied extensively by Langer in \cite{La, La2}.
	
	Let $C$ be a connected smooth projective curve defined 
	over an algebraically closed field $k$ of characteristic $p$, with $p \,>\, 0$. 
	Fix a locally free sheaf $E$ on $C$ --- of rank at least two --- and
	an integer $d \,\geqslant \, 2$. 
	Let ${\rm Quot}(E,d)$ denote the quot scheme parametrizing quotients of $E$
	of length $d$. For ease of notation, in this section, we shall denote ${\rm Quot}(E,d)$ by $\mc Q$. 
	In \cite{GS}, the authors compute the $S$-fundamental group scheme
	of $\mc Q$. In this section we summarize some of the results and ideas in that 
	article. 
	
	We mention some earlier results where fundamental
	group schemes were computed. In \cite{BH} it is proved that for a smooth projective 
	surface $X$, the \'etale fundamental group $\pi^{\et}(Hilb_X^n,\,nx)$ of the Hilbert 
	scheme of $n$ points ($n\,\geqslant  \, 2$), is isomorphic to $\pi^{\et}(X,\,x)_{\rm ab}$. The main 
	result in \cite{PS-surface} is to generalize this to the $S$-fundamental
	group scheme. In \cite{La2} it is proved that $\pi^S({\rm Alb}(C),\,0)$ 
	is isomorphic to $\pi^S(C,\,c)_{\rm ab}$. 
	Let $S_d$ ($d\,\geqslant \, 2$) be the permutation group of $d$ symbols
	and denote $C^{(d)} \,:=\, C^d/S_d$.
	In \cite{PS-curve} the authors 
	prove that the $\pi^S(C^{(d)},\,d[c])$ is isomorphic to $\pi^S(C,\,c)_{\rm ab}$.
	Once we have such a result for the $S$-fundamental group scheme, one deduces similar results 
	for the Nori and \'etale fundamental group schemes.
	
	\subsection{Hilbert-Chow morphism}\label{Hilbert-Chow}

	We recall the definition of the Hilbert Chow morphism 
	$\phi\,:\,\mc Q\,\longrightarrow\, C^{(d)}$. Let $p_1\,:\,C\times \mc Q\,\longrightarrow\, C$
	and $p_2\,:\,C\times \mc Q\,\longrightarrow\, \mc Q$ denote the natural projections,
	and let 
	\begin{equation}\label{hc-e1}
0\,\longrightarrow\, K\,\longrightarrow\, p_1^*E\,\longrightarrow\,\mc B\,\longrightarrow\, 0
	\end{equation}
	denote the universal short exact sequence on $C\times \mc Q$.
	Since $ \mc B$ and $p_1^*E$ are flat over $\mc Q$, it follows 
	that $K$ is a flat $\mc Q$ sheaf. Let $q\,\in\, \mc Q$ denote 
	a closed point. Restricting this quotient to $C\times \{q\}$ we get the exact sequence
	\begin{equation}\label{hc-e2}
0\,\longrightarrow\, K\vert_{C\times q}\,\longrightarrow\, E\,\longrightarrow\, \mc B\vert_{C\times q}
\,\longrightarrow\, 0\,. 
\end{equation}
	It follows that $K\vert_{C\times q}$ is a locally free 
	sheaf on $C$. From Nakayama's lemma it follows that 
	$K$ is a locally free $C\times \mc Q$ sheaf of rank $r\,:=\,{\rm rank}\, E$.
	Taking determinant of the inclusion in \eqref{hc-e1} we get an exact sequence
	$$0\,\longrightarrow\, {\rm det}(K)\,\longrightarrow\, {\rm det}(p_1^*E)\,\longrightarrow\, \mc F
\,\longrightarrow\, 0\,.$$
	To show that $\mc F$ is flat over $\mc Q$ it suffice 
	to prove that the restriction of this sequence to 
	$C\times q$ remains exact on the left. But this is clear
	as the restriction of this sequence to $C\times q$ 
	is precisely the sequence obtained by taking determinant
	of the inclusion in \eqref{hc-e2}, which remains exact on the left.
	Thus, on $C\times \mc Q$ we get a quotient 
	$$0\,\longrightarrow\, {\rm det}(K)\otimes{\rm det}(p_1^*E)^{-1}\,\longrightarrow\, \mc O
\,\longrightarrow\, \mc F\otimes {\rm det}(p_1^*E)^{-1}\,\longrightarrow\, 0\,.$$
	This defines a morphism
	\begin{equation}\label{defn-hc}
		\phi\,:\,\mc Q \,\longrightarrow\, C^{(d)}\,.
	\end{equation}
	In the following sections we describe properties of the fibers of this morphism.
	
	\subsection{Locus where $\phi$ is smooth}
	Consider the map $\phi\,:\,\mathcal{Q}\,\longrightarrow\, C^{(d)}$.
	Let $D$ denote the divisor $\sum^{k}_{i=1}d_{i}[c_{i}]$,
	and consider a quotient $q$ 
	\begin{equation}\label{thick-points}
		E\,\xrightarrow{\,\,\,q\,\,\,} \mathcal{O}_{D}\,\longrightarrow\, 0\,. 
	\end{equation}

	\begin{lemma}\label{GS-L3}
		The map $\phi\,:\,\mathcal{Q}\,\longrightarrow \,C^{(d)}$ 
		is smooth at $q$ (corresponding to the quotient in equation 
		\eqref{thick-points}).
	\end{lemma}

		This is proved in \cite{GS} by showing that the map of Zariski tangent spaces
		$T_{q}\mathcal{Q}\,\longrightarrow\, T_{\phi(q)}C^{(d)}$
		is surjective. 
	
	\begin{proposition}\label{Generic Smoothness}
		In every fiber of $\phi$ there is a point at which 
		$\phi$ is a smooth morphism.
	\end{proposition}

	\begin{proof}
		Let $D$ be the divisor corresponding to a point $x\,\in\, C^{(d)}$.
		Fix a line bundle $L$ which is a surjective quotient of 
		$E$. Then the composition of maps $E\,\longrightarrow\, L\,\longrightarrow\, L\otimes \mc O_D$,
which is denoted by $q$, is a quotient such that $\phi(q)\,=\,x$. The proposition now 
		follows from Lemma \ref{GS-L3}.
	\end{proof}

 \subsection{The scheme $S_d$ and fibers of the Hilbert Chow map}\label{section canonical bundle S_D}\label{S_d}

As before, $C$ is a complex projective curve and $E\, \longrightarrow\, C$
a vector bundle of rank $r$. Let $D\,\in \,C^{(d)}$. 
We fix an ordering of the points of $D$ 
\begin{equation}\label{define-sequence-c_i}
(c_1,\,c_2,\,\cdots,\,c_d)\, \in\, C^d\, .
\end{equation}
We will use this ordering to inductively construct a variety $S_j$ and a vector bundle
$A_j\, \longrightarrow\, C\times S_j$ for all $1\, \leqslant\, j\, \leqslant\, d$.

Set $S_{0}\,=\,{\rm Spec}\,\C$ and $A_{0}\,=\,E$. For $j\,\geqslant  \, 1$, we will define
$(S_{j},\, A_{j})$ assuming that the pair $(S_{j-1},\, A_{j-1})$ has been defined. 
Let 
$$\alpha_{j-1}\,:\,\{c_j\}\times S_{j-1}\,\hookrightarrow\, C\times S_{j-1} $$ 
be the natural closed immersion, where $c_j$ is the point in \eqref{define-sequence-c_i}.
Consider the projective bundle
\begin{equation}\label{def f_j,j-1}
f_{j,j-1}\,:\, S_{j}\,:=\,\mathbb{P}(\alpha^{*}_{j-1}A_{j-1})
\,\longrightarrow\, S_{j-1},
\end{equation}
and define the map
\begin{equation}
F_{j,j-1}\,:=\,{\rm Id}_C\times f_{j,j-1}\,:\,C\times S_{j}
\,\longrightarrow\, C\times S_{j-1}\,.
\end{equation}
We also have the natural closed immersion 
$$i_{j}\,:\,\{c_j\}\times S_{j}\,\hookrightarrow\, C\times S_{j}\,.$$
Let $p_{1,j}\, :\, C\times S_{j} \,\longrightarrow\, C$
and $p_{2,j}\, :\, C\times S_{j} \,\longrightarrow\, S_{j}$
be the natural projections. For each $j$, we have the following diagram
\[
\begin{tikzcd}
	\{c_j\}\times S_{j} \arrow[r,hook,"i_j"] \arrow[rd, "="]& C\times S_{j} \arrow[r,"F_{j,j-1}"] 
	\arrow[d,"p_{2,j}"] & C\times S_{j-1} \arrow[d,"p_{2,j-1}"] \\
	& S_j \arrow[r,"f_{j,j-1}"] & S_{j-1} \arrow[u, bend right=90, "\alpha_{j-1}", labels=below right]
\end{tikzcd}
\]
Let $\mathcal{O}_{j}(1) \,\longrightarrow\, S_{j}$ denote the universal line bundle.
Then over $C\times S_{j}$ we have the homomorphisms
\begin{align}\label{quotient-1}
F^{*}_{j,j-1}A_{j-1}\,\longrightarrow\, &\,\, (i_{j})_{*}i^{*}_{j}F^{*}_{j,j-1}A_{j-1} \\
	=\, &\,\, (i_{j})_{*}f^{*}_{j,j-1}\alpha_{j-1}^*A_{j-1} \nonumber\\
	\longrightarrow\, &\,\, (i_{j})_{*}\mathcal{O}_{j}(1)\nonumber
\end{align}
Define $A_{j}$ to be the kernel of the composition of
homomorphisms $$F^{*}_{j,j-1}A_{j-1}\,\longrightarrow\,(i_{j})_{*}\mathcal{O}_{j}(1)$$
in \eqref{quotient-1}. 
Thus, we have the following short exact sequence on $C\times S_j$
\begin{equation}\label{def-A_j}
	0\,\longrightarrow\, A_j\,\longrightarrow\, F_{j,j-1}^*A_{j-1}
	\,\longrightarrow\, i_{j*}(\mc O_j(1))\,\longrightarrow\, 0\,.
\end{equation}
For any $x\,\notin\, \{c_j\}\times S_{j} \,\subset\, C\times S_j$ the localizations of $A_j$ and $F_{j,j-1}^*A_{j-1}$ at $x$
are same. In particular, $A_j$ is locally free at $x$.
Now assume that $x\,\in\, \{c_j\}\times S_{j} \,\subset\, C\times S_j$,
and let $R$ be the local ring of $C\times S_j$ at $x$. Then the localization of $i_{j*}(\mc O_j(1))$ is given by $R/rR$
for some regular element $r\,\in\, R$ and hence its depth is given by 
${\rm dim}~R-1$. By Auslander--Buchsbaum formula we get that projective dimension of $i_{j*}(\mc O_j(1))$ is $1$. This implies that
$A_j$ is locally free at $x$. 
Therefore $A_j$ is locally free on $C\times S_j$. 
Thus, $(S_j,\, A_j)$ is constructed.

For $d\, \geqslant  \, j\,>\, i\,\geqslant  \, 0$, define morphisms
\begin{align}
f_{j,i}\,=\,f_{j,j-1}\circ \cdots \circ f_{i+1,i}\,& :\,S_{j}
\,\longrightarrow\, S_{i}, \label{f10}\\
F_{j,i}\,=\,{\rm Id}_C\times f_{j,i}\,=\,F_{j,j-1}\circ \cdots \circ F_{i+1,i}\,\,& :\,
C\times S_{j}\,\longrightarrow\, C\times S_{i}\,.\nonumber
\end{align}
Note that both the morphisms are flat. 

Closed points of $S_d$ are in bijective correspondence with the filtrations
\begin{equation*}
E_d\,\subset\, E_{d-1} \,\subset\, E_{d-2} \,\subset\, \cdots \,\subset\,
E_1\,\subset\, E_0\,=\,E,
\end{equation*}
where each $E_j$ is a locally free sheaf of rank $r$ on $C$
and $E_j/E_{j+1}$ is a skyscraper sheaf of degree one supported at 
$c_{j+1}\,\in \, C$. 

Let $p_1\,:\,C\times S_d\,\longrightarrow\, C$ denote the natural projection.
We will construct a quotient of $p_1^*E$.
Using the flatness of $F_{d,i}$, and pulling back \eqref{def-A_j},
for $i\,=\,0,\,\cdots,\,d-1$, we get a sequence of inclusions 
\begin{equation}\label{filtration on C times S_d}
A_{d}\,\subset\, F_{d,d-1}^{*}A_{d-1}\,\subset\, F_{d,d-2}^{*}A_{d-2}\,\subset\, \cdots 
\,\subset\, F_{d,1}^{*}A_{1}\,\subset\, F_{d,0}^*A_0\,=\,p^{*}_{1}E. 
\end{equation}
Define 
$$
B_{j}^d \,:=\, p^{*}_{1}E/F_{d,j}^{*}A_{j}\, \cong\, F_{d,j}^{*}(p^{*}_{1}E/A_j)\, ,
$$
so $B_{d}^d$ fits in the exact sequence
\begin{equation}\label{univ-quotient-S_D}
0\,\longrightarrow\, A_d\,\longrightarrow\, p^{*}_{1}E\,\longrightarrow\,
B^{d}_{d}\,\longrightarrow\, 0\,.
\end{equation}
The sheaf $B^{d}_{d}$ is $S_d$--flat.

Pulling back the exact sequence \eqref{def-A_j} on $C\times S_j$ along $F_{d,j}$,
and using flatness of $F_{d,j}$, we see that
$$F_{d,j-1}^{*}A_{j-1}/F_{d,j}^{*}A_j
\,\,\cong \,\,F_{d,j}^*(i_j)_*(\mc O_j(1))\,\,\cong\,\, f_{d,j}^*(\mc O_j(1))\,.$$ 
Thus, for each $j$ there is an exact sequence on $C\times S_d$
\begin{equation*}
0 \,\longrightarrow\, f_{d,j}^*(\mc O_j(1))\,\longrightarrow\, B_{j}^d
\,\longrightarrow\, B_{j-1}^d\,\longrightarrow\, 0\,.
\end{equation*}
The support of $B^d_d$ is finite over $S_d$. 

\begin{remark}
Alternatively, we can define the space $S_d$ as follows. 
Consider the nested Hilbert scheme ${\rm Quot}(E,d,d-1,\ldots,2,1)$. 
The closed points of this scheme are in bijection with quotients
$\{E\to B_d\to B_{d-1}\to \ldots \to B_1\}$, where each $B_j$ is a 
0-dimensional sheaf of length $j$, and all the arrows are surjections. 
There is a natural map ${\rm Quot}(E,d,d-1,\ldots,2,1)\to C^d$.
The space $S_d$ is precisely the fiber over the point 
$(c_1,\,c_2,\,\cdots,\,c_d)\, \in\, C^d$. \hfill \qedsymbol
\end{remark}

Consider the Hilbert-Chow map
\begin{equation}\label{hc}
\phi\,\,:\,\,\mc Q\,\longrightarrow\, C^{(d)}.
\end{equation} 
The universal property of $\mathcal{Q}$ and \eqref{univ-quotient-S_D} yields a morphism 
$g_d\,:\,S_d\,\longrightarrow\, \mathcal{Q}$. 
Let $\mc Q_D$ denote the scheme theoretic fiber of $\phi$ over the point
$D\,\in\, C^{(d)}$. It is straightforward to
check that $g_d$ factors as
\begin{equation}\label{f7}
g_d\,:\,S_d\,\longrightarrow\, \mc Q_D
\end{equation}
(the same notation $g_d$ is used for the map it is factoring through).
Define $V$ to be the open subset of $\mc Q_D$ consisting of all quotients of the form $E\,\longrightarrow\, \mc O_D$.
\begin{lemma}\label{the map g_d} ~
\begin{enumerate}
\item $g_d$ is surjective on closed points.

\item $g_d^{-1}(V) \,\longrightarrow\, V$ is an isomorphism.

\item $S_d\setminus g_d^{-1}(V) \,\longrightarrow\, \mc Q_D \setminus V$ has positive dimensional fibers.
\end{enumerate}
\end{lemma}

From this lemma, we can deduce the following
\begin{proposition}\label{fibres are of same dimension} \mbox{}
\begin{enumerate}
\item $\mc Q_{D}$ is irreducible, reduced and normal of dimension $d(r-1)$.

\item The Hilbert-Chow map $\phi$ is flat.
\end{enumerate}
\end{proposition}

Furthermore, we have the following lemma which allows us to reduce computations of cohomology of vector bundles on $\mc Q_D$ to 
cohomology of vector bundles on $S_d$.

\begin{lemma}\label{lemma-canonical bundle of S_D}
There is a line bundle $\mc L$ on $\mc Q_D$
such that the canonical bundle $K_{S_d}$ of $S_d$ is isomorphic to $g_d^*\mc L$.
\end{lemma}

\begin{corollary}\label{cor-cohomology over Q_D and S_D}
Let $V$ be a locally free sheaf on $\mc Q_D$. Then $H^i(\mc Q_D,\,V)\,\cong\,
H^i(S_d,\,g_d^*V)$.
\end{corollary}

\begin{proof}
As $g_d$ is birational (by Lemma \ref{the map g_d}), and $\mc Q_D$
is normal (by Proposition \ref{fibres are of same dimension}), it follows that 
$g_{d*}(\mc O_{S_d})\,=\,\mc O_{\mc Q_D}$. From \cite[Theorem 4.3.9]{Laz}
it follows that $R^qg_{d*}(K_{S_d})\,=\,0$ for all $q\,>\,0$. Using
Lemma \ref{lemma-canonical bundle of S_D}
this shows that $R^qg_{d*}(\mc O_{S_d})\otimes \mc L\,=\,0$ for $q\,>\,0$, that is,
$R^qg_{d*}(\mc O_{S_d})\,=\,0$ for $q\,>\,0$. Now the result follows using the 
Leray spectral sequence.
\end{proof}

\subsection{Fundamental group schemes}

\begin{definition}\label{cat-nf}
Let $X$ be a connected, projective and reduced $k$-scheme. Let $\mathcal{C}^{\rm nf}(X)$ denote the full subcategory of 
coherent sheaves whose objects are coherent sheaves $E$ on $X$ satisfying the following two conditions: 
\begin{enumerate}
\item $E$ is locally free, and 

\item for any smooth projective curve $C$ over $k$ and any morphism $f \,: \,C\, \longrightarrow \,X$, 
the vector bundle $f^*E$ is semistable of degree $0$. 
\end{enumerate}
\end{definition}

We call the objects of the category $\mc C^{\rm nf}(X)$ {\it numerically flat vector bundles} on $X$. 
Fix a $k$-valued point $x \,\in\, X$. 
	Let ${\rm Vect}_k$ be the category of finite dimensional $k$-vector spaces. 
	Let $T_x \,:\, \mc C^{\rm nf}(X) \,\longrightarrow\, {\rm Vect}_k$ be the fiber functor defined by 
	sending an object $E$ of $\mc C^{\rm nf}(X)$ to its fiber $E_x \,\in\, {\rm Vect}_k$ at $x$. 
	Then $(\mc C^{\rm nf}(X),\, \otimes,\, T_x,\, \mathcal{O}_X)$ is a neutral Tannakian category 
	\cite[Proposition 5.5, p.~2096]{La}.
	The affine $k$-group scheme $\pi^S(X,\, x)$ Tannakian dual to this category is called the 
	{\it S-fundamental group scheme} of $X$ with base point $x$ \cite[Definition 6.1, p.~2097]{La}. 
	
	A vector bundle $E$ is said to be {\it finite} if there are distinct non-zero polynomials 
	$f,\, g \,\in\, \mathbb{Z}[t]$ with non-negative coefficients such that $f(E) \,\cong\, g(E)$. 
	
	\begin{definition}
		A vector bundle $E$ on $X$ is said to be {\it essentially finite} if there exist two 
		numerically flat vector bundles $V_1,\, V_2$ and 
		finitely many finite vector bundles $F_1,\, \cdots,\ F_n$ on $X$ with 
		$V_2 \,\subseteq \,V_1 \,\subseteq\, \bigoplus\limits_{i=1}^n F_i$ such that $E \cong V_1/V_2$. 
	\end{definition}
	
	Let ${\rm EF}(X)$ be the full subcategory of coherent sheaves 
	whose objects are essentially finite vector bundles on $X$. 
	Fix a closed point $x \,\in\, X$ and let 
	$T_x \,: \,{\rm EF}(X) \,\longrightarrow\, {\rm Vect}_k$
	be the fiber functor defined by 
	sending an object $E \,\in\, {\rm EF}(X)$ to its fiber $E_x$ at $x$. 
	Then the quadruple $({\rm EF}(X),\, \bigotimes, \,T_x, \,\mathcal{O}_X)$ is a neutral Tannakian category. 
	The affine $k$-group scheme $\pi^N(X,\, x)$ Tannakian dual to this category is referred to as 
	the {\it Nori-fundamental group scheme} of $X$ with base point $x$, see 
	\cite{No1} for more details. 
	
	In \cite[Proposition 8.2]{La} it is proved that the $S$-fundamental group 
	of projective space is trivial. In \cite{Mehta-Hogadi} it is proved 
	that the $S$-fundamental group scheme is a birational invariant of
	smooth projective varieties. 

	Recall that for a divisor $D\,\in\, C^{(d)}$ we have fiber $\mc Q_D$ of the Hilbert-Chow map and also the space $S_{d}$
defined in section \ref{S_d}. By Lemma \ref{the map g_d},
	we have the birational map $g_d\,:\,S_d\,\longrightarrow\, \mc Q_D$.
	
\begin{proposition}
		A numerically flat bundle on $\mc Q_D$ is trivial.
	\end{proposition}

	\begin{proof}
		Let $W$ be a numerically flat bundle on $\mc Q_D$.
		As $\mc Q_D$ is normal, $g_d$ is birational and $S_{d}$ 
		is a smooth rational variety, we have 
		\begin{align*}
			W &\,\cong\, g_{d*}g_d^*W\\
			&\,\cong\, g_{d*}(\mc O)^{\oplus r}\\
			&\,\cong\, \mc O_{\mc Q_D}^{\oplus r}
		\end{align*}
		In the above we have used the result of \cite{Mehta-Hogadi}, which says that the $S$-fundamental group 
		scheme is a birational invariant. This proves the proposition. 
	\end{proof}
	
	\begin{theorem}\label{m-t}
		The induced map $\phi^S_*\,:\,\pi^S(\mc Q,\,q) \,\longrightarrow\, \pi^S(C^{(d)},\,\phi(q))$ is an isomorphism.
	\end{theorem}

	\begin{proof}
		Since the fibers of $\phi$ are projective integral varieties, and $\phi$ is flat, 
		it follows that $\phi_*(\mc O_{\mc Q_D})\,=\,\mc O_{C^{(d)}}$. Now applying \cite[Lemma 8.1]{La}
		we see that $\phi^S_*$ is faithfully flat. To prove $\phi^S_*$ is a closed immersion
		we will use \cite[Proposition 2.21(b)]{DMOS}. By Grauert's theorem 
		\cite[Corollary 12.9]{Ha} and the previous proposition, it follows that 
		if $W$ is a numerically flat bundle on $\mc Q$ then $\phi_*(W)$ is a locally 
		free sheaf on $C^{(d)}$ and the natural map $\phi^*\phi_*(W)\,\longrightarrow\, W$ is an 
		isomorphism. It follows easily that $\phi_*(W)$ is numerically flat. 
		This proves that $\phi^S_*$ is a closed immersion.
	\end{proof}

	From the $S$-fundamental group scheme we recover the Nori fundamental group scheme
	as the inverse limit of finite quotients. Similarly, the \'etale 
	fundamental group scheme can be recovered as the inverse limit of finite and 
	reduced quotients. Thus, we get the following corollary. (See \S~5.5 in \cite{PS-surface}
	for more details.)
	
	\begin{corollary}\label{m-c}
		The induced map $\phi^N_*\,:\,\pi^N(\mc Q,\,q) \,\longrightarrow\, \pi^N(C^{(d)},\,\phi(q))$ is an isomorphism.
		The induced map $\phi^{\et}_*\,:\,\pi^{\et}(\mc Q,\,q)\,\longrightarrow\, \pi^{\et}(C^{(d)},\,\phi(q))$ is an isomorphism.
	\end{corollary}
	
\section{NEF cone of ${\rm Quot}(\mc O_C^{\oplus n},d)$}

	Let $X$ be a smooth projective variety and let $N^1(X)$ 
	be the $\mb R$-vector space of $\mb R$-divisors modulo 
	numerical equivalence. It is known that $N^1(X)$ is a 
	finite dimensional vector space. The closed cone 
	${\rm Nef}(X)\,\subset\, N^1(X)$ is the cone of all 
	$\mb R$-divisors whose intersection product with 
	any curve in $X$ is non-negative. Similarly, let $N_1(X)$
	be the $\mb R$-vector space of 1-dimensional cycles
	modulo numerical equivalence. Let $\overline{NE}(X)\subset N_1(X)$
	denote the closure of the cone generated by classes 
	of effective curves.

	Let $E$ be a vector bundle over a smooth projective curve $C$. 
	Fix a polynomial $P\,\in \,\mb Q[t]$. Let ${\rm Quot}(E,P)$ denote 
	the Quot scheme parametrizing quotients of $E$ with Hilbert 
	polynomial $P$. In \cite{Str}, when $C\,=\,\mb P^1$, 
	the quot scheme ${\rm Quot}(\mc Q^{\oplus n}_{C},P)$ is 
	studied. In particular, its nef cone is described. 
	In \cite{GS-nef}, the authors attempt to understand the nef cone of the 
	Quot scheme ${\rm Quot}(\mc O_C^{\oplus n},d)$. 
	In this section, we describe some results and ideas in \cite{GS-nef}.
	For ease of notation, in this section, we 
	shall denote ${\rm Quot}(\mc O_C^{\oplus n},d)$ by $\mc Q$. 
	
	\subsection{Nef cone of $C^{(d)}$}\label{section-nef-C^{(d)}}
	
We follow \cite[\S~2]{Pa} for this section.
	Assume that either $C$ is an elliptic curve or is a very general curve of genus $g\,\geqslant \, 2$. 
	Then it is known that the N\'eron-Severi space is $2$-dimensional.
	So in this case, to compute the nef cone,
	it is enough to give two classes in $N^1(C)$ which are nef but not ample. 
	
	For any smooth projective curve and $d\,\geqslant  \,2$ 
	(not just a very general curve) 
	there is a natural line bundle $L_0$ on $C^{(d)}$ which is nef but not ample. 
	This line bundle is constructed in the following manner. Consider the map 
$$\phi\,:\,C^d 	\,\longrightarrow\, J(C)^{d \choose 2}\,,$$
	$$(x_i)\,\longmapsto\, (x_i-x_j)_{i<j}\,.$$
	Let $p_{ij}$ denote the projections from $J(C)^{d \choose 2}$. 
	Since $\phi$ is not finite, as it contracts the diagonal,
	the line bundle $\phi^*(\otimes p^*_{ij}\Theta)$ is nef but not ample. 
	This line bundle is invariant under the action of $S_d$ on $C^d$. 
	This follows from the fact that 
	$\Theta$ in $J(C)$ is invariant under the involution $L\,\longmapsto\, L^{-1}$. 
	\begin{definition}\label{def-L_0}
		$\phi^*(\otimes p^*_{ij}\Theta)$ descends to a line bundle $L_0$ on $C^{(d)}$. 
	\end{definition}

	Since $\phi$ contracts the small diagonal 
	$\delta\,:\,C\,\hookrightarrow \,C^{(d)}$, we have 
	$\delta^*[L_0]\,=\,0$. Hence $L_0$ is nef but not ample \cite[Lemma 2.2]{Pa}.
	Therefore, in the case where $C$ is very general, 
	computing the nef cone of $C^{(d)}$ boils down 
	to finding another class which is nef but not ample. 
	
	Recall that the gonality of $C$, denoted ${\rm gon}(C)$, is the smallest
	positive integer $e$ such that there is a degree $e$ morphism $C\, \longrightarrow\, \mb P^1$.
In the case where $d\,\geqslant \, {\rm gon}(C)\,=:\,e$, \cite[Lemma 2.3]{Pa},
we can easily construct another line bundle which is nef but not ample: 
	Then we have a map $g_e\,:\,C\,\longrightarrow\, \mb P^1$ of degree $e$. 
	This induces a closed immersion
$\mb P^1\,\longrightarrow\, C^{(e)}$ with $v\,\longmapsto\, [(g_e)^{-1}(v)]\,\in\, C^{(e)}$. 
This in turn gives a closed immersion
$\mb P^1\,\longrightarrow\, C^{(d)}$ with $v\,\longmapsto\, [(g_e)^{-1}(v)+(d-e)x]$ for some point $x\,\in\, C$. 

\begin{definition}\label{p1-in-C^(d)}
Denote the class of this $\mb P^1$ in $N_1(C^{(d)})$ by $[l']$. 
\end{definition}

The composition of maps $\mb P^1\,\longrightarrow\, C^{(d)}\,\xrightarrow{\,\,\,u_d\,\,\,}\, J(C)$ is a
constant one, 
since there can be no non-constant maps from $\mb P^1\,\longrightarrow\, J(C)$. 
Hence $u_d\,:\,C^{(d)}\,\longrightarrow\, J(C)$ is not finite and we get that 
$u_d^*\Theta$ is nef but not ample. 

\begin{definition}\label{def-theta_d}
Define $\theta_d\,\,:=\,\,u_d^*\Theta$.
\end{definition}

Recall that over $C^{(d)}$ we have following natural divisors \cite[\S~2]{Pa}:
	
\begin{definition}\label{three divisors of symmetric product of curves}\mbox{}
\begin{enumerate}
\item $\theta_d$;

\item the big diagonal $\Delta_d\,\hookrightarrow\, C^{(d)}$;

\item If $i_{d-1}\,:\,C^{(d-1)}\,\longrightarrow\, C^{(d)}$
is the map given by $D\,\longmapsto\, D+x$ for a point $x\,\in\, C$,
then the image $i_{d-1}(C^{(d-1)})$. This divisor will be denoted $[x]$. 
\end{enumerate}
\end{definition}
	
It is known that when $g\,=\,1$ or $C$ is very general of $g\,\geqslant \, 2$,
then $N^1(C^{(d)})$ is of dimension $2$ and any two of the 
above three forms a basis. 
	
By abuse of notation, let us denote the class ($\delta$ is the small diagonal) 
$[\delta_*(C)]\,\in\, N_1(C^{(d)})$
by $\delta$. We summarize the above discussion in the following proposition.
	
\begin{proposition}[{\cite[Proposition 2.4]{Pa}}]\label{nef-cone-C^(d)}
When $d\,\geqslant \, {\rm gon}(C)$, we have:
\begin{enumerate} 
\item ${\rm Nef}(C^{(d)})\,=\,\mb R_{\geqslant  0}[L_0]\oplus \mb R_{\geqslant  0}[\theta_d]$,

\item $\overline{NE}(C^{(d)})\,=\,\mb R_{\geqslant  0}[l']\oplus \mb R_{\geqslant  0}[\delta]$.
\end{enumerate}
The above two basis-es are dual to each other.
\end{proposition}

\begin{lemma}\label{E in terms of various bases}
Let $g,\,d\,\geqslant  \,1$. Let $\mu_0\,:=\,\dfrac{d+g-1}{dg}$. Then 
		\begin{align*}
			[L_0]&\,=\,dg [x]-[\theta_d]\\
			&=\,(dg-d-g+1).[x]+[\Delta_d/2]\\
			&=\,(\dfrac{1}{\mu_0}-1)[\theta_d]+\dfrac{1}{\mu_0}[\Delta_d/2].
		\end{align*}
	\end{lemma}
	
\subsection{Picard Group and N\'eron-Severi group of $\mc Q$}
	
	\begin{lemma}\label{pullback-lemma}
		Let $S$ be a scheme over $k$. Let $F$ be a 
		coherent sheaf over $C\times S$ which is $S$-flat 
		and for all $s\,\in \,S$, the restriction $F\big\vert_{C\times s}$ is a torsion 
		sheaf over $C$ of degree $d$. Let 
		$p_S\,:\,C\times S\,\longrightarrow\, S$ be the projection. Then the following two hold:
		\begin{enumerate}[(i)]
			\item $p_{S*}(F)$ is locally free of rank $d$ 
			and for all $s\,\in\, S$ the natural map 
			$p_{S*}(F)|_s\,\longrightarrow\, H^0(C,\,F|_{C\times s})$ is an isomorphism. 
			
\item Given a morphism $\phi\,:\,T \,\longrightarrow\, S$, there is the following diagram:
			\[
			\begin{tikzcd}
				C\times T \ar[r,"id\times \phi"] \ar[d,"p_T"] & C\times S \ar[d,"p_S"]\\
				T \ar[r,"\phi"] & S
			\end{tikzcd}
			\]
			Then the natural morphism
			$$\phi^*p_{S*}(F)\,\longrightarrow\, (p_T)_*(id\times \phi)^*F$$ is an isomorphism.
		\end{enumerate}
	\end{lemma}
	
We define a line bundle on $\mathcal{Q}$.
	Recall the projections $C\times \mathcal{Q}$ 
	to $C$ and $\mathcal{Q}$ by $p_{1}$ and $p_{2}$, respectively. 
	Then we have the universal quotient 
	$p^*_CE \,\longrightarrow\,\mc B$ over $C\times \mc Q$. By Lemma 
	\ref{pullback-lemma}, the direct image $p_{2*} ( \mc B) $ is a vector bundle of rank 
	$d$.

	\begin{definition}\label{def-O(1)}
		Denote the line bundle 
		{\rm det($p_{2*} ( \mc B) $)} by 
		$\mc {O}_{\mc {Q}}(1)$.
	\end{definition}
	
	As an easy Corollary to Lemma \ref{pullback-lemma} we have the following.
	
	\begin{lemma}\label{lb-lemma}
		Suppose we are given a map $f\,:\,T\,\longrightarrow\, \mc Q$. Let 
		$(id\times f)^* \mc B\,=\, \mc B_T$. Let $p_{1,T}\,:\,C\times T \,\longrightarrow\, C$ and 
		$p_{2,T}\,:\,C\times T \,\longrightarrow\, T$ be the projections. 
		\[
		\begin{tikzcd}
			C\times T \ar[r,"id\times f"] \ar[d,"p_{2,T}"] & C\times \mc Q \ar[d,"p_{2,\mc Q}"] \\
			T \ar[r,"f"] & \mc Q
		\end{tikzcd} 
		\]
		\begin{enumerate}[(i)]
			\item $f^*p_{2,\mc Q*}( \mc B\otimes p^*_C F)\, \longrightarrow\, p_{2,T*}( \mc B_T\otimes p^*_{1,T}F)$ is an isomorphism.
			\item For a vector bundle $F$ on $C$ define 
			$B_{F,T}\,:=\, {\rm det}(p_{2,T*}(\mc B_T\otimes p^*_{1,T}F))$. Then $f^*B_{F,\mc Q}\,=\,B_{F,T}$.
		\end{enumerate}
	\end{lemma}

	Recall that we have a Hilbert-Chow map $\phi\,:\,\mc Q \,\longrightarrow\, C^{(d)}$ (see Definition \ref{defn-hc}). 
	For notational convenience, for a divisor $D\,\in\, N^1(C^{(d)})$ we will denote 
	its pullback $\phi^*D\,\in\, N^1(\mc Q)$ by
	$D$, when there is no possibility of confusion.
	The Picard group of $\mc Q$ is well-known (see \cite[Theorem 3.7]{GS-nef}). 
	
\begin{theorem}
		${\rm Pic}(\mathcal{Q})\,\,=\,\,\phi^*{\rm Pic}(C^{(d)})\bigoplus \Z[\mathcal{O}_{\mathcal{Q}}(1)] $.
	\end{theorem}

	As a corollary we get that
	$N^1(\mc Q)\,\cong\, N^1(C^{(d)})\oplus \mb R[\mc O_{\mc Q}(1)]$.
	
	Let $\underline{c}\,=\,(c_1,\,c_2,\,\cdots,\,c_n)\,\in\, C^{(d)}$ be a point outside the big diagonal, that is,
	$c_i\,\neq\, c_j$ if $i\,\neq\, j$. 
	Recall the Hilbert-Chow map $\phi$ from \eqref{defn-hc}. Then 
	$$\phi^{-1}(\underline{c})\,\cong\, \prod \mb P(E_{c_i})\,.$$ 
	Let $\mb P^1\,\hookrightarrow\, \mb P(E_{c_1})$ be a line 
	and let $v_i\,\in\, \mb P(E_{c_i})$ for $i\,\geqslant \, 2$. Then we have an embedding:
	\begin{equation}\label{line-fiber-Phi}
\mb P^1\,\cong\, \mb P^1\times v_2\times\cdots\times v_d\,\hookrightarrow\,
\mb P(E_{c_1}) \times \prod\limits_{i\geqslant  2}\mb P(E|_{c_i})\,=\,\Phi^{-1}(\underline{c})\subset \mc Q.
	\end{equation}

	\begin{definition}\label{line-fiber-hc}
		Let us denote the class of this curve in $N_1(\mc Q)$ by $[l]$. 
	\end{definition}
	
We construct a section of $\phi\,:\,\mc Q\,\longrightarrow\, C^{(d)}$. 
Over $C\times C^{(d)}$ we have the universal divisor $\Sigma$ 
which gives us the universal quotient 
$\mc O_{C\times C^{(d)}}\,\longrightarrow\, \mc O_{\Sigma}$. Choose a surjection 
	$E\, \longrightarrow\, L$ over $C$, where $L$ is a line bundle on $C$. 
	This induces a surjection, obtained by pulling back along the first projection,
	$E\otimes \mc O_{C\times C^{(d)}}\,\longrightarrow\, L\otimes \mc O_{C\times C^{(d)}}$.
	Then the composition of homomorphisms
	$$E\otimes \mc O_{C\times C^{(d)}} \,\longrightarrow\, L\otimes \mc O_{C\times C^{(d)}}
\,\longrightarrow\, L\otimes \mc O_{\Sigma}$$ 
	gives us a morphism 
	\begin{equation}\label{section-HC}
		\eta\,:\,C^{(d)}\,\longrightarrow\, \mc Q .
	\end{equation}
	which is easily seen to be a section of $\phi$.
	We can show that 
		$$N_1(\mc Q)\,\,=\,\,N_1(C^{(d)})\oplus \mb R[l],$$
		where $N_1(C^{(d)})\,\hookrightarrow\, N_1(\mc Q)$ 
		is the morphism given by the push-forward $\eta_*$.

Let $p_1\,:\,C\times \mc Q \,\longrightarrow\, \mc Q$ and $p_{2}\,:\,C\times \mc Q \,\longrightarrow\, C$ 
be the natural projections. Recall the universal quotient $B$ on $C\times \mc Q$. 
	For a vector bundle $F$ over $C$, we define 
	$$B_{F,\mc Q}\,\,:=\,\,{\rm det}(p_{2*}(B\otimes p^*_CF)).$$
	Then we can show the following equality in $N^1(\mc Q)$
	\begin{equation}\label{compare-B_L-O(1)}
		[B_{L,\mc Q}]\,\,=\,\,[\mc O_{\mc Q}(1)]+{\rm deg}(L)[x].
	\end{equation}

\subsection{Upper bound on NEF cone}

Let $V$ be a vector space of dimension $n$. From now, unless mentioned otherwise, 
the notation $\mc Q$ will be reserved for the space 
$\mc Q(V\otimes \mc O_C,\,d)$. Sometimes we will also 
denote this space by $\mc Q(n,\,d)$ when we want to emphasize
$n$ and $d$. For the rest of this subsection, 
the genus of the curve $C$ will be $g(C)\,\geqslant \, 1$.
If $g(C)\,\geqslant \, 2$ then we will also assume that $C$ is very general. 

Our aim is to compute the NEF cone of $\mc Q$.
Since this cone is dual to the cone of effective
curves, it follows that if we take effective curves
$C_1,\,C_2,\,\cdots,\, C_r$, then taking the cone generated by these 
in $N_1(\mc Q)$, and then taking its dual cone $T$ in $N^1(\mc Q)$, we have the
following: the cone ${\rm Nef}(\mc Q)$ is contained in $T$. This gives us an upper 
bound on ${\rm Nef}(\mc Q)$.
We already know two curves in $\mc Q$. The first 
being a line in the fiber of $\phi\,:\,\mc Q \,\longrightarrow\, C^{(d)}$,
see Definition \ref{line-fiber-hc}, which was denoted $[l]$. 
Recall the section $\eta$ 
of $\Phi$ from equation \eqref{section-HC}, taking $L$ to be the trivial
bundle.
The second curve is $\eta_*([l'])$, where $[l']$ is from Definition 
\ref{p1-in-C^(d)}. Now we will construct a third curve in $\mc Q$.

Define a morphism 
\begin{equation}\label{t delta}
	\widetilde{\delta}\,:\,C \,\longrightarrow\, \mc Q
\end{equation}
as follows. Let $p_1,\,p_2\,:\,C\times C \,\longrightarrow\, C$ be the first and second projections respectively. 
Let $i\,:\,C \,\longrightarrow\, C\times C$ be the diagonal. 
Fix a surjection $k^n \,\longrightarrow\, k^d$ of vector spaces. 
Then define the quotient over $C\times C$
$$ \mc O^n_{C\times C} \,\longrightarrow\, \mc O^d_{C\times C} \,\longrightarrow\, i_*i^* \mc O^d_{C\times C}.$$ 
This induces a morphism $\widetilde{\delta}\,:\,C \,\longrightarrow\, \mc Q$ 
defined by $$c \,\longmapsto\, [\mc O^n_C \, \longrightarrow\, k^d_c\, \longrightarrow\, 0].$$
We will abuse notation and denote the class
$[\widetilde{\delta}_*(C)]\,in\, N_1(\mc Q)$ by $[\widetilde\delta]$.

We now give an upper bound for the NEF cone when $n\,\geqslant \, d\,\geqslant  \,{\rm gon}(C)$.

\begin{proposition}\label{nef cone of quot schemes } 
Consider the Quot scheme $\mc Q\,=\,\mc Q(n,\,d)$. Assume that $n\,\geqslant  \,d \,\geqslant \, {\rm gon}(C)$.
Let $\mu_0\,\,:=\,\,\dfrac{d+g-1}{dg}$. Then
$${\rm Nef}(\mc Q)\,\,\subset\, \,\mb R_{\geqslant  0}([\mc O_{\mc Q}(1)]+\mu_0[L_0])+\mb R_{\geqslant  0}[\theta_d]+\mb R_{\geqslant  0}[L_0].$$
\end{proposition}

\begin{proof}
	We claim that the cone dual to $\langle [l],\,\eta_*([l']),\,[\widetilde\delta]\rangle$
	is precisely 
\begin{equation}\label{z1}
\langle ([\mc O_{\mc Q}(1)]+\mu_0[L_0]),\,[L_0],\,[\theta_d]\rangle.
\end{equation}
To prove the claim, we have the following equalities:
	\begin{enumerate}
		\item $([\mc O_{\mc Q}(1)]+\mu_0[L_0])\cdot [l]\,=\,1$. This is clear.
		
\item $([\mc O_{\mc Q}(1)]+\mu_0[L_0])\cdot \eta_*[l']\,=\,0$. By 
		projection formula and \cite[Lemma 3.15]{GS-nef},
		we get that 
		$$([\mc O_{\mc Q}(1)]+\mu_0[L_0])\cdot [\eta_*l']\,=\,([-\Delta_d/2]+\mu_0[L_0])\cdot [l']\,.$$
		By Lemma \ref{E in terms of various bases} we get that 
		$[-\Delta_d/2]+\mu_0[L_0]\,=\,(1-\mu_0)[\theta_d]$. But as we saw earlier, 
		$[\theta_d]\cdot [l']\,=\, 0$.

		\item $([\mc O_{\mc Q}(1)]+\mu_0[L_0])\cdot [\widetilde{\delta}]\,=\,0$. By 
		Lemma \ref{pullback-lemma}, it is easy to see that 
		$[\mc O_{\mc Q}(1)]\cdot [\widetilde{\delta}]\,=\,0$. By 
		projection formula, we get 
		$$([\mc O_{\mc Q}(1)]+\mu_0[L_0])\cdot [\widetilde{\delta}]\,=\,
		[\mu_0L_0]\cdot [\Phi_*\widetilde{\delta}]\,=\,[\mu_0L_0]\cdot [\delta]=0\,.$$
		
\item $[\theta_d]\cdot [l]\,=\,[L_0]\cdot [l]\,=\, 0$ follows using the projection formula. 	
	\end{enumerate}
	Now the claim in \eqref{z1} follows from Proposition \ref{nef-cone-C^(d)}.

	As explained before, since ${\rm Nef}(\mc Q)$ is contained in the 
	dual to the cone $\langle [l],\,\eta_*([l']),\,[\widetilde \delta]\rangle$,
	the proposition follows.
\end{proof}

When the genus $g\,=\,1$, we have the following improvement of 
Proposition \ref{nef cone of quot schemes }, in which the condition 
that $n\,\geqslant  \,d$ is not required.

\begin{proposition}\label{ub-g=1} 
Let $C$ be a smooth projective curve of genus $g\,=\,1$.
Consider the Quot scheme $\mc Q\,=\,\mc Q(n,\,d)$.
Assume $d \,\geqslant \, {\rm gon}(C)\,=\,2$. Then 
	$${\rm Nef}(\mc Q)\,\,\subset\,\, \mb R_{\geqslant  0}([\mc O_{\mc Q}(1)]+[L_0])+\mb R_{\geqslant  0}[\theta_d]+\mb R_{\geqslant  0}[L_0].$$
\end{proposition}

The following two Lemmas will be useful in showing that certain line bundles are nef.

\begin{lemma}\label{B_L.D geq 0}
	Let $f\,:\,D\,\longrightarrow\, \mc Q$ be a morphism, where $D$ is a smooth projective curve. Fix a point 
	$q\,\in\, f(D)$ and an effective divisor $A$ on $C$ containing the scheme theoretic support of
	$\mc B_q$. If there is a line bundle $L$ on $C$ such that $H^0(C,\, L)\,\longrightarrow\, H^0(A,\, L\big\vert_A)$ is
surjective, then $[B_{L,\mc Q}]\cdot [D]\,\,\geqslant \,\, 0$. 
\end{lemma}

\begin{lemma}\label{line bundles restricted to a divisor}
Let $A$ be an effective divisor on $C$ of degree $d$.
Then there is a line bundle
$L$ on $C$ of degree $d+g-1$ such that the natural map 
$$H^0(C,\, L)\,\longrightarrow\, H^0(A,\, L\big\vert_A)$$
is surjective.
\end{lemma}

\subsection{Nef cone in the genus 0 case}\label{genus-0}

Throughout this section we will work with $C\,=\,\mb P^1$.
Let us first compute the nef cone of $\mc Q(n,\,d)$.

Note that we have $C^{(d)}\,\cong \,\mb P^d$. 
Hence $N^1(C^{(d)})\,=\,\mb R[\mc O_{\mb P^d}(1)]$. 
As $N^1(\mc Q)\,\cong\, N^1(C^{(d)})\oplus \mb R[\mc O_{\mc Q}(1)]$,
it follows that $N^1(\mc Q)$ is two dimensional. Therefore, it suffices
to find a line bundle on $\mc Q$ which is different from 
the pullback of $\mc O_{\mb P^d}(1)$ and which is nef but 
not ample. The following result is proved in \cite[Theorem 6.2]{Str}. 

\begin{proposition}
\begin{align*}
{\rm Nef}(\mc Q(n,\,d)) \,=\, & \mb R_{\geqslant  0}[B_{\mc O(d-1),\mc Q}]+\mb R_{\geqslant  0}[\mc O_{\mb P^d}(1)]\\
=\, & \mb R_{\geqslant  0}([\mc O_{\mc Q}(1)]+(d-1)[\mc O_{\mb P^d}(1)])+\mb R_{\geqslant  0}[\mc O_{\mb P^d}(1)]\,.
\end{align*}	
\end{proposition}

The following generalizes the above to all vector bundles on the projective line.

\begin{theorem}\label{nef cone of quot schemes over projective line}
	Let $C\,=\,\mb P^1$. Let $E\,=\,\bigoplus\limits^k_{i=1} \mc O(a_i)$ with $a_i\,\leqslant\, a_j$ for $i\,<\,j$. 
	Let $d\,\geqslant \, 1$ and $L\,=\,\mc O(-a_1+d-1)$.
	Then 
	\begin{align*}	
		{\rm Nef}(\mc Q(E,\,d)) \,=\, & \mb R_{\geqslant  0}[B_{L,\mc Q(E\,d)}]+\mb R_{\geqslant  0}[\mc O_{\mb P^d}(1)] \\
		=\, & \mb R_{\geqslant  0}([\mc O_{\mc Q(E,d)}(1)]+(-a_1+d-1)[\mc O_{\mb P^d}(1)])+\mb R_{\geqslant  0}[\mc O_{\mb P^d}(1)].
	\end{align*}
\end{theorem}

\begin{proof}
	It is enough to give two line bundles which are nef but not ample. 
	Clearly $\Phi^*_{\mc Q(E,\,d)}\mc O_{\mb P^d}(1)$ is nef but not ample.
	So it suffices to show that $B_{L,\mc Q(E,d)}$ is nef but not ample.
	
	Since $a_j-a_1\,\geqslant \, 0$ for every $j\,\geqslant \, 1$, we get that $E(-a_1)$ is globally generated. Let 
	$V\,:=\,H^0(C,\, E(-a_1))$ and $\dim\, V\,=\,n$. Then we have a surjection 
	$V\otimes \mc O_C \,\longrightarrow\, E(-a_1)$. This gives a surjection
	$$V\otimes \mc O_C\, \longrightarrow\, p^*_CE(-a_1)\,\longrightarrow\, \mc B_{\mc Q(E,d)}\otimes p^*_C\mc O_C(-a_1)
\,\longrightarrow\, 0.$$
	The above surjection defines a map $f\,:\,\mc Q(E,\,d)\,\longrightarrow\, \mc Q(n,\,d)$. By Lemma \ref{lb-lemma} we get that 
	$$f^*B_{\mc O(d-1),\mc Q(n,d)}\,=\,B_{L,\mc Q(E,d)}\,=\,{\rm det}(p_{\mc Q(E,d)*}(\mc B_{\mc Q(E,d)}\otimes p_C^*L)).$$ 
	Since $B_{\mc O(d-1),\mc Q(n,d)}$ is nef, we get that $B_{L,\mc Q(E,d)}$ is nef.
	We next show that the $B_{L,\mc Q(E,d)}$ is not ample. 
	Consider the section $\eta_{\mc Q(E,d)}$ of 
	$\Phi_{\mc Q(E,d)}\,:\,\mc Q(E,\,d)\,\longrightarrow\, C^{(d)}$ defined by the quotient 
	$p_C^*E\,\longrightarrow\, p_C^*\mc O(a_1)\otimes \mc O_{\Sigma}$ on $C\times C^{(d)}$
	(see (\ref{section-HC})).
	Then $f\circ \eta_{\mc Q(E,d)}$ is a section of $\Phi\,:\,\mc Q(n,\,d)\,\longrightarrow\, C^{(d)}$ 
	defined by a quotient $\mc O^n_C\,\longrightarrow\, \mc O_\Sigma \,\longrightarrow\, 0$ on $C\times C^{(d)}$. Therefore, 
	$\eta_{\mc Q(E,d)}^*B_{L,\mc Q(E,d)}\,=\,\eta^*B_{\mc O(d-1),\mc Q(n,d)}$.
	As $\eta^*B_{\mc O(d-1),\mc Q(n,d)}$ is not ample,
	we get that $B_{L,\mc Q(E,d)}$ is not ample.
	The second equality follows again from the fact that $[x]\,=\,[\mc O_{\mb P^d}(1)]$.
\end{proof}

\subsection{Nef cone in the genus 1 case}\label{subsection-criterion-for-nefness}

Now we are back to the assumption that 
the genus of the curve satisfies $g(C)\,\geqslant \, 1$
and if $g(C)\,\geqslant \, 2$, then we also assume that $C$ is very general. 

\begin{definition}\label{def-U'}
	Let $U'\,\subset\, \mc Q$ be the open set consisting of 
	quotients $\mc O^n_C\,\longrightarrow\,B \,\longrightarrow\, 0$ such that the induced 
	map $H^0(C,\,\mc O^n_C)\,\longrightarrow\, H^0(C,\,B)$ is surjective.
\end{definition}

\begin{lemma}\label{the open set V}
Consider the Quot scheme $\mc Q\,=\,\mc Q(n,\,d)$.
Let $D$ be a smooth projective curve, and let 
	$D \,\longrightarrow\, \mc Q$ be a morphism such that its image intersects $U'$. Then 
$$[\mc O_{\mc Q}(1)]\cdot [D]\,\,\geqslant \,\, 0.$$
\end{lemma}

\begin{proof}
Let $p_D\,:\,C\times D \,\longrightarrow\, D$ be the natural projection. Then applying 
	$(p_D)_*$ to the quotient $\mc O^n_{C\times \mc Q}\,\longrightarrow\, \mc B_D$ we get that the morphism
	$$(p_D)_*\mc O^n_{C\times D}\,=\,\mc O^n_D \,\longrightarrow\, (p_D)_*\mc B_D$$
	is generically surjective by our assumption and Lemma \ref{pullback-lemma}. Hence it follows that 
	$$[\mc O_{\mc Q}(1)]\cdot [D]\,\,=\,\,{\rm deg}((p_D)_*\mc B_D)\,\geqslant \, 0.$$
This completes the proof. 
\end{proof}

In the remainder of this subsection, we will need to work with
$C^{(d)}$ for different values of $d$. The line bundles $L_0$ 
over $C^{(d)}$ will therefore be denoted by $L_0^{(d)}$ when 
we want to emphasize the $d$. Similarly, we will denote 
$\mu_0^{(d)}\,=\,\dfrac{d+g-1}{dg}$.
Let $\mc P^{\leqslant n}_{(d)}$ be the set of all partitions 
$(d_1,\,d_2,\,\cdots,\,d_k)$ of $d$ of length $k\,\leqslant\, n$. 
Given an element $\mathbf{d}\,\in\, \mc P^{\leqslant n}_{(d)}$, define 
$$C^{(\mathbf{d})}\,\,:=\,\,C^{(d_1)}\times C^{(d_2)}\times \cdots \times C^{(d_k)},$$ 
and if $p_{i,\mathbf{d}}\,:\,C^{(\mathbf{d})}\,\longrightarrow\, C^{(d_i)}$ is the $i$-th projection,
we define a class
$$[\mc O(-\Delta_{\mathbf{d}}/2)]\,\,:=\,\,\left[\sum p_{i,\mathbf{d}}^*\mc O(-\Delta_{d_i}/2)\right] \,\in\,
N^1(C^{(\mathbf{d})})$$
Note that we have a natural addition 
$$\pi_{\mathbf{d}}\,\,:\,\,C^{(\mathbf{d})} \,\,\longrightarrow\,\, C^{(d)}.$$
For a partition $\mathbf{d}\,\in\, \mc P^{\leqslant n}_d$ define a morphism 
$$\eta_{\mathbf{d}}\,:\,C^{(\mathbf{d})}\,\longrightarrow\, \mc Q$$
as follows. For any $l\,\geqslant \, 1$, define the universal 
subscheme of $C^{(l)}$ over $C\times C^{(l)}$ by $\Sigma_l$. 
Then over $C\times C^{(\mathbf{d})}$ we have the subschemes 
$(id\times p_{i,\mathbf{d}})^*\Sigma_{d_i}$. We have a quotient 
$$q_{\mathbf{d}}\,:\,\mc O^n_{C\times C^{(\mathbf{d})}}
\,\longrightarrow\, \bigoplus\limits_i \mc O_{(id\times p_{i,\mathbf{d}})^*\Sigma_{d_i}}$$
defined by taking direct sum of morphisms 
$\mc O_{C\times C^{(\mathbf{d})}}\,\longrightarrow\, \mc O_{(id\times p_{i,\mathbf{d}})^*\Sigma_{d_i}}$. 
Then $q_{\mathbf{d}}$ defines a map $C^{(\mathbf{d})}\,\longrightarrow\, \mc Q$. 

Using G\"ottsche's theorem (\cite[page 9]{Pa}), we can show the following Lemma.

\begin{lemma}\label{lemma-pullback of O(1) via various sections}
	Let $\eta_{\mathbf{d}}\,:\,C^{(\mathbf{d})} \,\longrightarrow\, \mc Q$ be the map defined by
	$q_{\mathbf{d}}$. Then
	$$[\eta^*_{\mathbf{d}}\mc O_{\mc Q}(1)]\,\,=\,\,[\mc O(-\Delta_{\mathbf{d}}/2)].$$
\end{lemma}

\begin{definition}\label{def-U}
Define $U\,\subset\, \mc Q$ to be the set of 
quotients of the form 
$$
\mc Q^{\oplus n}_{C}\, \longrightarrow\, \dfrac{\mc O_{C}}{\prod_{i=1}^r\mf m^{d_i}_{C,c_i}} \,\cong \,
\bigoplus \dfrac{\mc O_{C,c_i}}{\mf m^{d_i}_{C,c_i}}\qquad\qquad c_i \neq c_j. $$
\end{definition}

The following Lemma says that some intersection numbers 
with curves which meet $U$ are non-negative. 

\begin{lemma} \label{the open set U}
	Consider the Quot scheme $\mc Q\,=\,\mc Q(n,\,d)$.
	Let $D$ be a smooth projective curve and let 
	$D\,\longrightarrow\, \mc Q$ be a morphism such that its image intersects $U$. Then 
	$([\mc O_{\mc Q}(1)]+[\Delta_d/2])\cdot [D]\geqslant  0$.
\end{lemma} 

\begin{lemma}\label{bound on O(1) in terms of diagonals}
	Let $D$ be a smooth projective curve.
	Let $D\, \longrightarrow\, \mc Q$ be a morphism. Then there exists a 
	partition $\mathbf{d}\,\in\, \mc P^{\leqslant n}_{(d)}$ (which depends on the morphism $D\, \longrightarrow\, \mc Q$) 
	such that the following two conditions hold:
	\begin{enumerate}
		\item the composition of maps $D\, \longrightarrow\, \mc Q\, \longrightarrow\, C^{(d)}$ factors as 
		$D\, \longrightarrow\, C^{(\mathbf{d})}\, \longrightarrow\, C^{(d)}$, and
		\item $[\mc O_{\mc Q}(1)]\cdot [D]\,\geqslant \, [\mc O(-\Delta_{\mathbf{d}}/2)]\cdot [D]$. 
	\end{enumerate} 
\end{lemma}
 
\begin{proof}
	We will proceed by induction on $d$. When $d=1$ the statement is obvious.
	
	Let us denote the pullback of the universal quotient on $C\times \mc Q$ to $C\times D$ by $\mc B_D$ and let $f\,
:\,C\times D\, \longrightarrow\, D$ be the natural projection. 	
	Consider a section such that the composite $\mc O_{C\times D}\, \longrightarrow\, \mc O_{C\times D}^n\, \longrightarrow\, \mc B_D$
	is non-zero and let $\mc F$ denote the cokernel of the composite map.
	We have a commutative diagram
	\begin{equation}\label{diagram-inductive-step}
		\xymatrix{
			0\ar[r] & \mc O_{C\times D}\ar[r]\ar@{->>}[d] & \mc O_{C\times D}^n \ar[r]\ar@{->>}[d] & 
			\mc O_{C\times D}^{n-1}\ar[r]\ar@{->>}[d] & 0\\
			0\ar[r] & \mc O_{\Gamma'}\ar[r] & \mc B_D\ar[r] & \mc F\ar[r] & 0
		}
	\end{equation}
	Let $T_0(\mc F)\,\subset\, \mc F$ denote the maximal subsheaf of dimension $0$, see
	\cite[Definition 1.1.4]{HL}. Define $\mc F'\,:=\,\mc F/T_0(\mc F)$. Now, either $\mc 
	F'\,=\,0$ or $\mc F'$
	is torsion free over $D$, and hence, flat over $D$. 
	In the first case, it follows that $D$ meets the open set $U$
	in Definition \ref{def-U}. Then we take $\mathbf{d}=(d)$ 
	and the statement follows from Lemma \ref{the open set U}. 
	So we assume $\mc F'$ is flat over $D$ and let $d'$ be 
	the degree of $\mc F'|_{C\times x}$, for $x\,\in\, D$. So 
	$0\,<\,d'\,<\,d$. By (\ref{diagram-inductive-step})
	we have
	$${\rm deg}~f_*\mc B_D\,=\, {\rm deg}~f_*\mc O_{\Gamma'}+{\rm deg}~f_*\mc F\,.$$
	Since $T_0(\mc F)$ is supported on finitely many points, we have 
	${\rm deg}~\mc F\,\geqslant \, {\rm deg}~\mc 
	F'$. In other words, we have 
	\begin{equation}\label{inequality required for the inductive step}
		{\rm deg}~f_*\mc B_D\,\geqslant \, {\rm deg}~f_*\mc O_{\Gamma'}+f_*\mc F'	 \,.	
	\end{equation}
	Now $\Gamma'$ defines a morphism $D\, \longrightarrow\, C^{(d-d')}$ and note that
	$${\rm deg}~f_*\mc O_{\Gamma'}\,=\,[\mc O(-\Delta_{d-d'}/2)]\cdot [D]\,.$$
	The quotient $\mc O^{n-1}_{C
\times D}\, \longrightarrow\, \mc F' \, \longrightarrow\, 0$ defines a map $D\, \longrightarrow\, \mc Q(n-1,\,
d')$. By induction 
	hypothesis, we get that there exists a partition 
	$\mathbf{d}'\,\in\, \mc P^{\leqslant n-1}_{d'}$ such
	that the composition of maps $D\, \longrightarrow\, \mc Q(n-1,\,d')\, \longrightarrow\, C^{(d')}$ factors as
	$D\, \longrightarrow\, C^{(\mathbf{d'})}\, \longrightarrow\, C^{(d')}$ and 
	$$[\mc O_{\mc Q(n-1,d')}(1)]\cdot [D]\,\geqslant \, [\mc O(-\Delta_{\mathbf{d'}}/2)]\cdot [D]\,.$$
	Since ${\rm deg}~f_*\mc F'\,=\,[\mc O_{\mc Q(n-1,d')}(1)]\cdot [D]$ we have that 
	${\rm deg}~f_*\mc F'\geqslant  [\mc O(-\Delta_{\mathbf{d'}}/2)]\cdot [D]$.
	From (\ref{inequality required for the inductive step}) we get that 
	$$[\mc O_{\mc Q}(1)]\cdot [D]\,\geqslant \, [\mc O(-\Delta_{d-d'}/2)]\cdot D + [\mc O(-\Delta_{\mathbf{d'}}/2)]\cdot [D]\,.$$
	Now we define $$\mathbf{d}\,\,:=\,\,(d-d',\,\mathbf{d'})$$ and 
	the statement follows from the above inequality. 
\end{proof}

\begin{theorem}\label{criterion-for-nefness}
	Let $\beta\in N^1(C^{(d)})$.
	Then the class $[\mc O_{\mc Q}(1)]+\beta \,\in\, N^1(\mc Q)$ is nef if and only if the class 
	$[\mc O(-\Delta_{\mathbf{d}}/2)]+\pi^*_{\mathbf{d}}\beta\,\in\, N^1(C^{(\mathbf{d})})$ 
	is nef for all $\mathbf{d}\,\in\, \mc P^{\leqslant n}_d$.
\end{theorem}

\begin{proof}
	From Lemma \ref{lemma-pullback of O(1) via various sections} it is clear that if 
	$[\mc O_{\mc Q}(1)]+\beta$ is nef, then 
	$\eta^*_{\mathbf{d}}([\mc O_{\mc Q}(1)]+\beta)\,=\,[\mc O(-\Delta_{\mathbf{d}}/2)]+\pi^*_{\mathbf{d}}\beta$ is nef.
	
	For the converse, we assume 
	$[\mc O(-\Delta_{\mathbf{d}}/2)]+\pi^*_{\mathbf{d}}\beta$ is 
	nef for all $\mathbf{d}\,\in\, \mc P^{\leqslant n}_d$. Let $D$ be 
	a smooth projective curve and $D\, \longrightarrow\, \mc Q$ be a morphism.
	By Lemma \ref{bound on O(1) in terms of diagonals} we have that there exists $\mathbf{d}\in \mc P^{\leqslant n}_d$ such that 
	$D\, \longrightarrow\, C^{(d)}$ factors as $D\, \longrightarrow\, C^{(\mathbf{d})}\, \longrightarrow\, C^{(d)}$ and 
	$$[\mc O_{\mc Q}(1)]\cdot [D]\,\,\geqslant \,\, [\mc O(-\Delta_{\mathbf{d}}/2)]\cdot[D]\,.$$
	Now by assumption we have that 
	$$[\mc O(-\Delta_{\mathbf{d}}/2)]\cdot[D]\,\,\geqslant \,\, -\beta\cdot [D]\,.$$
	Therefore we get 
	$$[\mc O_{\mc Q}(1)]\cdot [D]\,\,\geqslant \,\, -\beta\cdot [D]\,.$$
	Hence we get that the class $[\mc O_{\mc Q}(1)]+\beta$ is nef.
\end{proof}

\begin{lemma} \label{compare various L_0}
	Suppose we are given a map $D\, \longrightarrow\, C^{(\mathbf{d})} \,\xrightarrow{\,\,\,\pi_{\mathbf{d}}\,\,\,}\, C^{(d)}$.
	Then we have
	$$[L^{(d)}_0]\cdot [D]\,\,\geqslant \,\, \sum_{i}[L^{(d_i)}_0]\cdot [D]\,.$$ 
\end{lemma}

\begin{proof}
	By $[L_0^{d_i}]\cdot [D]$ we mean the degree of the pullback 
	of $[L_0^{(d_i)}]$ along $D\, \longrightarrow\, C^{(\mathbf{d})}\,\xrightarrow{\,\,\,p_i\,\,\,}\, C^{(d_i)}$.
	The lemma follows easily from the definition of $L_0^{(d)}$ 
	and is left to the reader. 
\end{proof}

\begin{proposition}\label{another lower bound for NEF cone}
	Let $n\,\geqslant \, 1$, $g\,\geqslant  \,1$ and $\mc Q\,=\,\mc Q(n,\,d)$.
	Then the class $\kappa_2\,:=\,[\mc O_{\mc Q}(1)]+\dfrac{g+1}{2g}[L_0^{(d)}] \,\in\, N^1(\mc Q)$ is nef.
\end{proposition}

\begin{proof}
	Recall $\mu_0^{(2)}=\dfrac{g+1}{2g}$.
	By Theorem \ref{criterion-for-nefness}
	it suffices to show that for all $\mathbf{d}\,\in\, \mc P^{\leqslant n}_{(d)}$
	we have $[\mc O(-\Delta_{\mathbf{d}}/2)]+\mu_0^{(2)}\pi^*_{\mathbf{d}}[L_0^{(d)}]$ is nef.
	Using Lemma \ref{E in terms of various bases}, 
	$[L^{(1)}_0]\,=\,0$ and Lemma \ref{compare various L_0} we get 
	\begin{align*}
		([\mc O(-\Delta_{\mathbf{d}}/2)]+\mu_0^{(2)}\pi^*_{\mathbf{d}}[L_0^{(d)}])\cdot [D]&=
		\Big(\sum_i(1-\mu_0^{(d_i)})[\theta_{d_i}]-\mu_0^{(d_i)}[L_0^{d_i}]\Big)\cdot [D] \\
		&\qquad \qquad +\mu_0^{(2)}[L_0^{(d)}]\cdot [D]\\
		&\geqslant  \sum_i(\mu_0^{(2)} - \mu_0^{(d_i)})[L_0^{d_i}]\cdot [D]\,.
	\end{align*}
	This proves that $\kappa_2$ is nef.	
\end{proof}

\begin{corollary}\label{O(1)+mu^2L_0-not-ample}
	Let $n\geqslant  d$. Then the class $[\mc O_{\mc Q}(1)]+\mu^{(2)}_0[L^{(d)}_0]\in N^1(\mc Q)$ is nef but not ample.
\end{corollary}
\begin{proof}
	By Proposition \ref{another lower bound for NEF cone} we have that 
	$[\mc O_{\mc Q}(1)]+\mu^{(2)}_0[L^{(d)}_0]$ is nef. Now recall that
	when $n\geqslant  d$ we have the curve $\widetilde \delta\,\hookrightarrow\, \mc Q$ (\ref{t delta}).
	From the definition of $\widetilde \delta$ and Lemma \ref{lb-lemma} we have 
	$[\mc O_{\mc Q}(1)]\cdot [\widetilde \delta]\,=\,0$. Also $\Phi_*\widetilde \delta\,=\,\delta$. Hence 
	$[L^{(d)}_0]\cdot [\widetilde \delta]\,=\,[L^{(d)}_0]\cdot [\delta]\,=\,0$. From this we get 
	$[\mc O_{\mc Q}(1)]+\mu^{(2)}_0[L^{(d)}_0]\cdot [\widetilde \delta]\,=\,0$ and hence 
	$[\mc O_{\mc Q}(1)]+\mu^{(2)}_0L^{(d)}_0$ is not ample.
\end{proof}

As a corollary of Proposition \ref{another lower bound for NEF cone}
we get the following result. When $g\,=\,1$ note that $\mu_0^{(d)}\,=\,\mu_0^{(2)}\,=\,1$. 

\begin{theorem}\label{O(1)+Delta/2 is nef for elliptic curves}
	Let $g\,=\,1$, $n\,\geqslant \, 1$ and $\mc Q\,=\,\mc Q(n,\,d)$.
	Then the class $[\mc O_{\mc Q}(1)]+[\Delta_d/2]\,\in\, N^1(\mc Q)$ is nef.
	Moreover, 
	$${\rm Nef}(\mc Q)\,\,=\,\, \mb R_{\geqslant  0}([\mc O_{\mc Q}(1)]+[\Delta_d/2])+
	\mb R_{\geqslant  0}[\theta_d]+\mb R_{\geqslant  0}[\Delta_d/2]\,.$$
\end{theorem}
\begin{proof}
	This follows using Lemma \ref{E in terms of various bases}, Proposition \ref{nef cone of quot schemes }
	and the above Corollary.
\end{proof}

\subsection{Nef cone in the genus $g\geqslant   2$ case}

When $g\geqslant  2$, there are some partial results on the nef cone of $\mc Q$. 
We only state these, and refer the reader to \cite{GS-nef} for proofs.

One extremal ray in ${\rm Nef}(C^{(2)})$ is given by $L_0$.
There is a real number $t$ such that 
the other extremal ray of ${\rm Nef}(C^{(2)})$ be given by 
\begin{equation}\label{def-alpha_t}
	\alpha_t\,\,=\,\,(t+1)x-\Delta_2/2\,,
\end{equation} 
(see \cite[page 75]{Laz}). Then using Lemma \ref{E in terms of various bases}, we get that
\begin{equation}\label{relation between L_0, Delta, alpha}
	\Delta_2/2\,\,=\,\,\dfrac{t+1}{g+t}L_0-\dfrac{g-1}{g+t}\alpha_t\,.
\end{equation}

\begin{theorem}\label{cone-d=2}
	Let $d\,=\,2$. Consider the Quot scheme $\mc Q\,=\,\mc Q(n,\,2)$. 
	Then $${\rm Nef}(\mc Q)\,\,=\,\,\mb R_{\geqslant  0}([\mc O_{\mc Q}(1)]+\dfrac{t+1}{g+t}[L_0])+
	\mb R_{\geqslant  0}[L_0]+\mb R_{\geqslant  0}[\alpha_t].$$ 
\end{theorem}

\begin{theorem}\label{cone-d=3}
	Let $C$ be a very general curve of genus $2\,\leqslant \,g(C)\,\leqslant\, 4$. 
	Let $n\,\geqslant \,3$ and let $\mc Q\,=\,\mc Q(n,\,3)$. Let $\mu_0\,=\,\dfrac{g+2}{3g}$
	Then 
	$${\rm Nef}(\mc Q)\,\,= \,\,\mb R_{\geqslant  0}([\mc O_{\mc Q}(1)]+\mu_0[L_0])+
	\mb R_{\geqslant  0}[\theta_d]+\mb R_{\geqslant  0}[L_0]\,.$$
\end{theorem}

\section{Automorphisms and Deformations of $\mc Q$}

Let ${\rm Quot}(E,d)$ denote the quot scheme parametrizing quotients of $E$
of length $d$. For ease of notation, in this section, we shall denote ${\rm Quot}(E,d)$ by $\mc Q$. 
Recall that the group of holomorphic automorphisms $\text{Aut}(\mc Q)$ of $\mc Q$ is a complex Lie group
whose Lie algebra is the Lie algebra of vector fields $H^0(\mc Q,\,T_{\mc Q})$ \cite[Lemma 1.2.6]{Sern}. Here $T_{\mc Q}$ is the tangent bundle of $\mc Q$. 
Also, the space of first order infinitesimal deformations of $\mc Q$ is given by $H^1(\mc Q, \, T_{\mc Q})$. In this section we discuss the computations of this two spaces,
and more generally the sheaf cohomology of $\mc Q$ following \cite{BGS}. 

When $E\,=\,\mc O^{\oplus r}$ the lie algebra of vector fields was computed in \cite{BDH-aut}
$$H^0(\mc Q({\mathcal O}^{\oplus r},\,d),\,
T_{\mc Q({\mathcal O}^{\oplus r},d)})\,=\, \mf{sl}(r, {\mathbb C})\,=\,
H^0(X,\, {\rm End}({\mathcal O}^{\oplus r}))/{\mathbb C}
$$
for all $r\, \geqslant \, 2$. From this it follows that the maximal
connected subgroup of $\text{Aut}(\mc Q({\mathcal O}^{\oplus r},\,d))$ is
${\rm PGL}(r,{\mathbb C})\,=\, \text{Aut}({\mathcal O}^{\oplus r})/{\mathbb C}^*$.
This was generalized in \cite{G19}. If either $E$ is semistable or $r\,\geqslant  \, 3$,
it was proved that
$$
H^0(\mc Q(E,\,d),\, T_{\mc Q(E,d)})\,\,=\,\, H^0(X,\, \text{End}(E))/{\mathbb C}
$$
\cite{G19}, and hence the maximal connected subgroup of $\text{Aut}(\mc Q(E,\,d))$
is $\text{Aut}(E)/{\mathbb C}^*$. 
Regarding the next cohomology $H^1(\mc Q,\,T_{\mc Q})$, first consider the
case of $r\,=\,1$. In this case, the Quot scheme
$\mc Q(E,\,d)$ is identified with the $d$-th symmetric product 
$C^{(d)}$ of $C$. The infinitesimal deformation space
$H^1(C^{(d)},\, T_{C^{(d)}})$ was computed in \cite{Kempf} under
the assumption that $C$ is 
non-hyperelliptic, and it was computed in \cite{Fantechi} when 
$g\,\geqslant  \, 3$ (see also \cite[Remark 2.6]{Fantechi} for the case $g=2$). 

Henceforth, we will always assume that $r\, =\, {\rm rank}(E)\, \geqslant  \, 2$.

If $d\,=\,1$, then $\mc Q\,\,\cong\,\, \mb P(E)$. Consequently,
$H^1(\mc Q(E,\,1),\,T_{\mc Q(E,1)})$ can be computed easily.

Associated to the vector bundle $E$ there is the Atiyah bundle $At(E)$ on $C$
\cite[Theorem 1]{Atiyah}. The infinitesimal 
deformations of the pair $(C,\, E)$ are parametrized by
$H^1(X,\, At(E))$ \cite[Proposition 4.2]{Chen}. For the natural
homomorphisms $\mc O_C\,\hookrightarrow \,\text{End}(E)
\,\hookrightarrow \, At(E)$, the quotients $\text{End}(E)/\mc O_C$
and $At(E)/\mc O_C$ are vector bundles and will be denoted by $ad(E)$ and $at(E)$ respectively. 
Also, given any vector bundle $V$ on $C$ we can 
construct a natural bundle called the Secant bundle 
${Sec}^d(V)$ on $C^{(d)}$ (see \cite[Proposition 1]{Mattuck}, \cite[Section 2]{biswas-laytimi}). 
In particular we have a bundle $Sec^d(at(E))$ on $C^{(d)}$.

Recall that we have the Hilbert-Chow map $\phi\,:\,\mc Q\,\longrightarrow\, C^{(d)}$. We refer to 
Subsection \ref{Hilbert-Chow} for the definition of this map. 
Then we have

\begin{theorem}\label{theorem 1}
Let $r\,=\, {\rm rank}(E)\,\geqslant   \,2$. Then
\begin{enumerate}
\item $Sec^d(at(E))\,\cong\, \phi_*T_{\mc Q}$ and

\item $R^i\phi_*T_{\mc Q}\,=\,0$ for all $i\,>\,0$.
\end{enumerate}
\end{theorem}

The cohomologies of $Sec^d(at(E))$ can be computed easily. Therefore, using the Leray Spectral Sequence
we can compute the spaces $H^i(\mc Q, \,T_{\mc Q})$ using Theorem \ref{theorem 1}.

\begin{theorem}\label{theorem 2} 
Let ${\rm rank}(E),\, d\,\geqslant  \, 2$. Denote the genus of $C$ by $g_C$.
The following three statements hold:
\begin{enumerate}
\item For all $d-1\, \geqslant  \, i\, \geqslant  \, 0$,
\begin{align*}
H^i(\mc Q,\,T_{\mc Q})\,=\,H^0(C,\,at(E))\otimes & \bigwedge^ i H^1(C,\,\mc O_{C})\\
& \bigoplus H^1(C,\,at(E))\otimes 
\bigwedge^{i-1}H^{1}(C,\,\mc O_{C})\, .
\end{align*}
In particular,
$$\dim H^i(\mc Q,\, T_{\mc Q})\,=\, {g_C \choose i} \cdot \dim H^0(C,\,at(E))+
{g_C \choose {i-1}} \cdot \dim H^1(C,\, at(E))\,.$$

\item When $i\,=\,d$,
$$H^d(\mc Q, \, T_{\mc Q})\, =\, \bigwedge^{d-1}H^1(C,\,\mc O_C) \otimes h^1(C, \, at(E))\,.$$
In particular,
$$h^d(\mc Q, \, T_{\mc Q})\, =\, {g_C \choose {d-1}} \cdot \dim H^1(C, \, at(E))\,.$$

\item
For all $i\,\geqslant \, d+1$, $$H^i(\mc Q,\, T_{\mc Q})\,=\,0.$$
\end{enumerate}
\end{theorem}

Recall that $\mc Q$ is a fine moduli space, that is, 
there exists a certain universal quotient on $C\times \mc Q$. 
The kernel of this universal quotient, which is locally free, is denoted
by $\mc A$. Recall that the space $H^1(C\times \mc Q,\,{\ms End}(\mc A))$
is the space of all infinitesimal
deformations of $\mc A$. We have the following result.

\begin{theorem}\label{theorem 3}
Let ${\rm rank}(E),\,g_C,\,d\,\geqslant  \, 2$. Then we have
$$H^1(C\times \mc Q,\,{\ms End}(\mc A))
\,=\,H^1(C\times \mc Q,\,\mc O_{C\times \mc Q})
\,=\,H^1(C,\,\mc O_C)\oplus H^1(\mc Q,\,\mc O_{\mc Q})\,.$$
\end{theorem}

It can be seen using Corollary \ref{cor-direct image of structure sheaf} 
that $H^1(\mc Q,\,\mc O_{\mc Q})\,=\,H^1(C,\,\mc O_{C})$.

In \cite{G18} it was shown that when $E\,\cong \,\mc O^n_C$ 
for some $n\,\geqslant   \,1$, then $\mc A$ is slope stable with 
respect to some natural polarizations of $C\times \mc Q$. 
By \cite[Corollary 4.5.2]{HL}, the space $H^1(C\times \mc Q,\,{\ms End}(\mc A))$ is the tangent 
space at $[\mc A]$ of the Moduli space $\mc M$ of sheaves on 
$C\times \mc Q$ with the same Hilbert polynomial 
(with respect to some fixed polarization on $C\times \mc Q$) as $\mc A$. 
Moreover, the differential of the determinant map $\mc M\,\longrightarrow\, {\rm Pic}(C\times \mc Q)$ at the point $[\mc A]$ is given by 
the trace map 
$$H^1(C\times \mc Q,\,{\ms End}(\mc A))\,\,\xrightarrow{\,\,\,H^1(tr)\,\,\,}\,\, H^1(C\times \mc Q,\, \mc O_{C\times \mc Q})$$
(see \cite[Theorem 4.5.3]{HL}).
Note that this map is onto, since the composition of the maps $$\mc O_{C\times \mc Q}
\, \longrightarrow\,{\ms End}(\mc A) \,\,\xrightarrow{\,\,\,tr\,\,\,}\,\, \mc O_{C\times \mc Q}$$ is an isomorphism. 
Therefore, Theorem \ref{theorem 3} implies 
that the determinant map $\mc M\,\longrightarrow\, {\rm Pic}(C\times \mc Q)$
induces an isomorphism at the level of tangent spaces 
at the point $[\mc A]$.

In the following subsections, we discuss the proof of Theorem \ref{theorem 1}. 
The strategy is the following: We use Cohomology and Base change theorems to compute the direct images of Tangent bundle under the Hilbert-Chow map 
(Recall that the Hilbert-Chow map is flat by Lemma \ref{the map g_d}) This reduces the problem to computing cohomology of various vector bundles on the fibers of the Hilbert-Chow map.
Now by Lemma \ref{the map g_d} , there exists a birational morphism from a tower of projective bundles to these fibers. We show that this resolution is in fact a rational resolution
of the fibers (Lemma \ref{lemma-canonical bundle of S_D}), so in particular computing cohomology of a vector bundle on the fiber is same as computing the cohomology of its pullback to this resolution (Corollary \ref{cor-cohomology over Q_D and S_D}). Hence the problem ultimately reduces to computing cohomology of vector bundles on tower of projective bundles, which can be done using various known cohomology computations on projective bundles.

\subsection{Atiyah sequence}

Let $V$ be a locally free sheaf of rank $r$ over a smooth variety $X$.
Its Atiyah bundle $At(V)\, \longrightarrow\, X$ fits
in the following \textit{Atiyah exact sequence}
\begin{equation}\label{Atiyah-seq}
0\,\longrightarrow\, \ms End(V)\,\longrightarrow\, At(V)
\,\longrightarrow\, T_X\,\longrightarrow\, 0
\end{equation}
(see \cite{Atiyah}). Here $\mathscr{E}nd(V)$ is the sheaf of local endomorphisms of the bundle $V$.
We recall a construction of \eqref{Atiyah-seq} which will be
used. Let $P_V\,\stackrel{q}{\longrightarrow}\,X$
denote the principal ${\rm GL}_r(\C)$-bundle associated to $V$.
The differential of $q$ produces an exact sequence on $P_V$
\begin{equation}\label{f1}
0\,\longrightarrow\, K\,:=\, T_{P_V/X}\,\longrightarrow\, T_{P_V}
\,\stackrel{dq}{\longrightarrow}\, q^*T_X\,\longrightarrow\, 0
\end{equation}
Applying $q_*$ to it and then taking ${\rm GL}_r(\C)$-invariants we get
\eqref{Atiyah-seq}.

We have ${\mathcal O}_X\, \subset\, \ms End(V)$; the quotient
$ad(V)\,:=\, \ms End(V)/{\mathcal O}_X$ is identified with the sheaf of
endomorphisms of $V$ of trace zero. Define $at(V)\,=\, At(V)/{\mathcal O}_X$.
Taking the push-out of \eqref{atiyah-seq}
along the quotient map $\ms End(V)\,\longrightarrow\, ad(V)$ we get an exact sequence
\begin{equation}\label{atiyah-seq}
0\,\longrightarrow\, ad(V)\,\longrightarrow\, at(V)\,\longrightarrow\, T_X
\,\longrightarrow\, 0\,.
\end{equation}

Next we define the relative Atiyah sequence. 
Let $X$ be a smooth projective variety, $C$ a smooth projective curve
and $V$ a vector bundle on $C\times X$. Let 
\begin{equation}\label{p2}
p_C\,:\,C\times X\,\longrightarrow\, C\ \ \text{ and }\ \ p_X\,:\,C\times X
\,\longrightarrow\, X
\end{equation} 
be the natural projections. Pulling back, along the inclusion map
$p_C^*T_C\,\hookrightarrow\, p_C^*T_C \oplus p_X^*T_X$, 
of the Atiyah exact sequence for $V$
\begin{equation}\label{f4}
0\,\longrightarrow\, {\ms End}(V)\,\longrightarrow\, At(V)
\,\longrightarrow\, p_C^*T_C \oplus p_X^*T_X\,\longrightarrow\, 0
\end{equation}
we get the relative Atiyah sequence
\begin{equation}\label{eqn-relative atiyah exact seq}
0 \,\longrightarrow\, {\ms End}(V)\,\longrightarrow\, At_C(V)
\,\longrightarrow\, p_C^*T_C \,\longrightarrow\, 0 \,.
\end{equation}
The push-out of \eqref{eqn-relative atiyah exact seq} along the projection
$\ms End(V)\,\longrightarrow\, ad(V)$ produces an exact sequence
\begin{equation}\label{eqn rel ad-Atiyah}
0\,\longrightarrow\, ad(V)\,\longrightarrow\, at_C(V)
\,\longrightarrow\, p_C^*T_C\,\longrightarrow\, 0
\end{equation}
on $C\times X$.
Henceforth, \eqref{eqn rel ad-Atiyah} will be referred to as the {\it relative
adjoint Atiyah} sequence. The following lemma shows that the 
relative adjoint Atiyah sequence is stable under base change.

\begin{lemma}\label{base change relative ad-atiyah} Let $f\,:\,Y\,\longrightarrow\, X$ be a morphism of smooth 
projective varieties. Define
$$F\,:=\,{\rm Id}_C\times f\,:\,C\times Y\,\longrightarrow\, C\times X.$$ 
The relative adjoint Atiyah sequence for $F^*V$ coincides with the
one obtained by applying $F^*$ to \eqref{eqn rel ad-Atiyah}.
\end{lemma}

Let
$$
[At_C(V)]\, \in\, {\rm Ext}^1(p_C^*T_C,\,{\ms End}(V))
$$
be the class of \eqref{eqn-relative atiyah exact seq}.
\begin{lemma}\label{cor-compatibility of atiyah classes}
Let $f\,:\,V\,\longrightarrow\, V'$ be a morphism between vector bundles
$V,\, V'$ on $X$. Then the image of $[At_C(V)]$ under the natural map
$${\rm Ext}^1(p_C^*T_C,\,{\ms End}(V))\,\,
\xrightarrow{\,\,\,f\circ \_\,\,}\,\, {\rm Ext}^1(p_C^*T_C,\,{\ms Hom}(V,\,V'))$$
coincides with the image of $A(V')$ under the natural map
$${\rm Ext}^1(p_C^*T_C,\,{\ms End}(V'))\,\xrightarrow{\,\,\,\_ \circ f\,\,\,} 
\,{\rm Ext}^1(p_C^*T_C,\,{\ms Hom}(V,\,V'))\,.$$
\end{lemma}

\subsection{Cohomology of some sheaves on $S_d$}\label{section cohomology of some sheaves}

Recall from Section \ref{section canonical bundle S_D} that associated to a divisor $D\,\in\, 
C^{(d)}$ and an ordering $(c_1,\,c_2,\,\cdots,\,c_d)$ of points of $D$ we have the schemes $S_j$ 
for $1\,\leqslant\, j\,\leqslant\, d$. In this section we again fix this divisor $D$. We now choose the 
ordering $(c_1,\,c_2,\,\cdots,\,c_d)$ of $D$ in the following manner: Define $c_1$ to be any 
point in $D$. Now suppose we have defined $c_j$ for $1 \,\leqslant\, j\,\leqslant\, d-1$. Then we 
define $c_{j+1}$ to be $c_j$ if $c_j\,\in \,D\setminus \{c_1,\,c_2,\,\cdots,\,c_j\}$. Otherwise 
define $c_j$ to be any point in $D\setminus \{c_1,\,c_2,\,\cdots,\,c_j\}$. Throughout this 
section, we fix this ordering of $D$. To clarify, the above conditions on the ordering do 
not determine the ordering uniquely.

Recall from \eqref{univ-quotient-S_D} that we have an exact sequence 
$$0\,\longrightarrow\, A_d \,\longrightarrow\, F^*_{d,0}A_0\,\longrightarrow\,
B^d_d\,\longrightarrow\, 0$$
on $C\times S_d$, and a filtration
$$A_{d}\,\subset\, F_{d,d-1}^{*}A_{d-1}\,\subset\, F_{d,d-2}^{*}A_{d-2}\,\subset\,
\cdots\,\subset\, F_{d,1}^{*}A_{1}\,\subset\, F_{d,0}^*A_0\,=\,p^{*}_{1}E$$
(see \eqref{filtration on C times S_d}).
Since $B^d_d$ is supported on $D\times S_d$, there is an inclusion map
$$F^*_{d,0}A_0(-D)\,\subset\, A_d\,.$$
We list some cohomology computations of various vector bundles associated
to these bundles which we will need later. All of these follow from applying repeatedly the computations of cohomology of the relative cotangent bundle and the universal line bundle of a projective bundle.

\begin{lemma}\label{cor-cohomology of A_d}
For any $t\,\in \, C$, the following statements hold:
\begin{enumerate}
\item the natural map $H^i\left(S_d,\,F^*_{d,0}A_0(-D\times S_d)\big\vert_{t\times S_d}\right)
\,\longrightarrow\, H^i\left(S_d,\,A_d\big\vert_{t\times S_d}\right)$ is an isomorphism,

\item the natural map $H^i\left(S_d,\,F^*_{d,0}A_0^{\vee}\big\vert_{t\times S_d}\right)
\,\longrightarrow\, H^i\left(S_d,\,A_d^{\vee}\big\vert_{t\times S_d}\right)$ is an isomorphism,

\item $H^i\left(S_d,\,A_d\big\vert_{t\times S_d}\right)\,=\,0\,\ \ \forall\,\ i\,>\,0$, and

\item $H^i\left(S_d,\,A_d^{\vee}\big\vert_{t\times S_d}\right)\,=\,0\,\ \ \forall\,\ i\,>\,0$.
	\end{enumerate} 
\end{lemma}

For convenience of notation, denote $G_d\,:=\, A_d\big\vert_{c_d\times S_d}$.

\begin{lemma}\label{lemma-cohomology computation end(G_d)}
Using the above notation,
\begin{enumerate}
\item $h^1(S_d,\,\ms End(G_d))\,=\,1$,

\item $\dim H^i(S_d,\,\ms End(G_d))\,=\,0$ for all $i\,\geqslant  \,2$,

\item $\dim H^i(S_d,\,G_d^\vee(1))\,=\,0$ for all $i\,\geqslant  \,1\,.$
\end{enumerate}
\end{lemma}

\subsection{On the geometry of $\mc Q$}

\subsubsection{Notation}\label{se9.1}

As before, the $d$-th symmetric product of $C$ is denoted by $C^{(d)}$. 
Let 
\begin{align}
p_1\,:\,C\times \mc Q\,\longrightarrow\, C,\ \ p_2\,:\,
C\times \mc Q\,\longrightarrow\, \mc Q,\label{t5}\\
q_1\,:\, C\times C^{(d)}\,\longrightarrow\, C,\ \ 
q_2\,:\,C\times C^{(d)}\,\longrightarrow\, C^{(d)}\label{t6}
\end{align}
denote the natural projections.
Recall that there is a universal exact sequence on $C\times \mc Q$
\begin{equation}\label{universal-seq-C times Q}
0 \,\longrightarrow\, \mc A\,\longrightarrow\, p_1^*E\,\longrightarrow\, \mc B
\,\longrightarrow\, 0\,.
\end{equation}
Let
\begin{equation}\label{f13}
\Sigma\,\subset\, C\times C^{(d)}
\end{equation}
be the universal divisor.
Define
\begin{equation}\label{t7}
\Phi\,:=\,{\rm Id}_C\times \phi\,:\,C\times \mc Q\,\longrightarrow\, C\times C^{(d)},
\end{equation}
where $\phi$ is the Hilbert-Chow map in \eqref{hc}.

\subsubsection{Direct image of sheaves on $C\times \mc Q$}

\begin{corollary}\label{cor-direct image of structure sheaf}
The following statements hold:
\begin{enumerate}
\item $\Phi_*\mc O_{C\times \mc Q}\,=\, \mc O_{C\times C^{(d)}}$\ and\ 
$R^i\Phi_*\mc O_{C\times \mc Q}\,=\,0\,\ \ \forall\ i\,>\,0$.

\item $\phi_*\mc O_{\mc Q}\,=\,\mc O_{C^{(d)}}$\ and\ $R^i\phi_*\mc O_{\mc Q}
\,=\,0\,\ \ \forall\ i\,>\,0$.
\end{enumerate}
\end{corollary}

\begin{proof}
The fibers of $\phi$ (respectively, $\Phi$) over any point $D\,\in\, C^{(d)}$
(respectively, $(c,\, D)\,\in\, C\times C^{(d)}$) is isomorphic to $\mc Q_D$.
By Corollary \ref{cor-cohomology over Q_D and S_D} we have 
$\dim H^i(\mc Q_D,\,\mc O_{\mc Q_D})\,=\,\dim H^i(S_d,\,\mc O_{S_d})$. 
Since $S_d$ is a tower of projective bundles, it 
follows that $\dim H^0(S_d,\,\mc O_{S_d})\,=\,1$ 
and $\dim H^i(S_d,\,\mc O_{S_d})\,=\,0$ for all 
$i\,>\,0$. As both $\phi$ and $\Phi$ are flat 
morphisms, the result now follows from Grauert's theorem 
\cite[p.~288--289, Corollary 12.9]{Ha}.
\end{proof}

Let
\begin{equation}\label{f16}
\mc Z \, \subset\, C\times \mc Q
\end{equation}
be the zero scheme of the inclusion map 
${\rm det}(\mc A)\,\hookrightarrow\, \det (p_1^*E)$, where $p_1$ is the map
in \eqref{t5}. From the
definition of $\phi$ it follows immediately that $\Phi^*\Sigma\,=\,\mc Z$. 
In fact, $\mc Z$ sits in the following commutative diagram
in which both squares are Cartesian
\begin{equation}\label{f14}
\xymatrix{
	\mc Z\ar[r]\ar[d] & C\times \mc Q\ar[r]\ar[d]^\Phi & \mc Q\ar[d]^\phi \\
	\Sigma \ar[r] & C\times C^{(d)}\ar[r] & C^{(d)}
}
\end{equation}
(see \eqref{f13} and \eqref{f16}) and the composition of the top horizontal maps is a finite morphism;
the same holds for the composition of the bottom horizontal maps in \eqref{f14}.
The ideal sheaf $\mc O_{C\times \mc Q}(-\mc Z)$ therefore
annihilates $\mc B$ in \eqref{universal-seq-C times Q}, which
in turn produces an inclusion
map $p_1^*E(-\mc Z)\,\subset\, \mc A$, where $p_1$ is the map
in \eqref{t5}. Applying $\Phi_*$ 
and using Corollary \ref{cor-direct image of structure sheaf} an inclusion map
\begin{equation}\label{f15}
q_1^*E(-\Sigma)\,\cong\, \Phi_*[p_1^*E(-\mc Z)] \,\hookrightarrow\, \Phi_*\mc A
\end{equation}
is obtained, where $q_1$ is the map in \eqref{t6}. Also, note that
since the cokernel of $\mc A\, \longrightarrow\, p_1^*E$ is $\mc B$, the kernel of the map $p^*_1E^{\vee} \, \longrightarrow\,
\mc A^{\vee}$ is $\mc B^{\vee}$. But $\mc B$ is torsion, hence $\mc B^{\vee}\,=\,0$.
 Therefore we also have the natural inclusion map
$p_1^*E^{\vee}\,\hookrightarrow\, \mc A^{\vee}$. Applying $\Phi_*$ and
using Corollary \ref{cor-direct image of structure sheaf} we get an inclusion map
$$q_1^*E^{\vee} \,\hookrightarrow\, \Phi_*(\mc A^{\vee})\,.$$

\begin{proposition}\label{propn-direct image of A}
The following statements hold:
\begin{enumerate}
\item The natural map $q_1^*E(-\Sigma)\,\longrightarrow\, \Phi_*\mc A$ is an isomorphism
(see \eqref{f15}).

\item The natural map $q_1^*E^{\vee}\,\hookrightarrow\,\Phi_*(\mc A^{\vee})$ is an isomorphism.

\item $R^i\Phi_*\mc A\,=\,R^i\Phi_*(\mc A^{\vee})\,=\,0$\ for all\, $i\,>\,0$.
\end{enumerate} 
\end{proposition}

\begin{proof}
First consider the map $\Phi_*[p_1^*E(-\mc Z)]\,\longrightarrow\, \Phi_*\mc A$.
Fix $(c,\,D)\,\in \,C\times C^{(d)}$. We will show that the homomorphism
\begin{equation}\label{f12}
H^0\left(c\times \mc Q_D,\, p_1^*E(-\mc Z)\big\vert_{c\times \mc Q_D}\right)
\,\longrightarrow\, H^0\left(c\times \mc Q_D,\, \mc A\big\vert_{c\times \mc Q_D}\right)
\end{equation}
is an isomorphism.

In view of Lemma \ref{cor-cohomology over Q_D and S_D},
showing that \eqref{f12} is an isomorphism is equivalent 
to showing that the map 
\begin{equation*}
H^0\left(c\times S_d,\, p_1^*E(-D)\big\vert_{c\times S_d}\right)\,\longrightarrow\,\,
H^0\left(c\times S_d,\, A_d\big\vert_{c\times S_d}\right)
\end{equation*}
is an isomorphism. But this is precisely the content of 
Corollary \ref{cor-cohomology of A_d} (1) when $i\,=\,0$.
Since $\mc A$ is flat over $C\times C^{(d)}$, using
Grauert's theorem, \cite[Corollary 12.9]{Ha}, it now follows that 
$q_1^*E(-\Sigma)\,\longrightarrow\, \Phi_*\mc A$ is an isomorphism. 

The other two statements follow from 
Corollary \ref{cor-cohomology of A_d} in the same way. 
\end{proof}

\begin{corollary}\label{Phi_*B}
The natural map 
$$q_1^*E\big\vert_\Sigma\,\longrightarrow\, \Phi_*\mc B$$
is an isomorphism, where $\Sigma$ is defined in \eqref{f13}
and $q_1$ (respectively, $\Phi$) is the map in \eqref{t6}
(respectively, \eqref{t7}). Moreover, $$R^i\Phi_*\mc B\,=\,0$$
for all $i\,>\,0$.
\end{corollary}

\begin{proof}
Using projection formula and Corollary \ref{cor-direct image of structure sheaf} it
follows that $$\Phi_*p_1^*E\,=\,q_1^*E\ \ \text{ and }\ \
R^i \Phi_* p_1^*E\,=\,q_1^*E\otimes R^i\Phi_*\mc O_{C\times \mc Q}\,=\,0$$
for all $i\,>\,0$.
Therefore the statement follows immediately by applying $\Phi_*$ to the 
universal exact sequence 
$$0\,\longrightarrow\,\mc A\,\longrightarrow\, p_1^*E
\,\longrightarrow\, \mc B\,\longrightarrow\, 0$$
and using Proposition \ref{propn-direct image of A}.
\end{proof}

\begin{lemma}\label{Hom(B,B)}
The natural map $\mc O_{\mc Z}\,\longrightarrow\, \ms Hom(\mc B,\,\mc B)$,
where $\mc Z$ is defined in \eqref{f16}, is an isomorphism.
\end{lemma}

\begin{proof}
It can be shown that $\mc Z$ is an integral and normal scheme and that $\mc B$ in \eqref{universal-seq-C times Q} is a 
torsionfree 
$\mc O_{\mc Z}$-module.

Let ${\rm Spec}(A)\,\subset\, \mc Z$ be an affine open set,
and let $\mc B_A$ denote the module corresponding to $\mc B$. 
We have the inclusion maps 
$$A\,\subset\, {\rm Hom}_A(\mc B_A,\,\mc B_A)\,\subset\, 
{\rm Hom}_{K(A)}(\mc B_A\otimes_AK(A),\,\mc B_A\otimes_AK(A))\,=\,K(A)\,.$$
As $A$ is normal, and ${\rm Hom}_A(\mc B_A,\,\mc B_A)$ 
is a finite $A$-module, it follows that ${\rm Hom}_A(\mc B_A,\,\mc B_A)$ coincides
with $A$. This proves the lemma.
\end{proof}

\begin{theorem}\label{preliminiary ses}
There is a map $\Xi$ that fits in a short exact sequence
\begin{equation*}
0\,\longrightarrow\, ad(q_1^*E\big\vert_{\Sigma})\,
\stackrel{\Xi}{\longrightarrow}\, \Phi_*\ms Hom(\mc A,\,\mc B)
		\,\longrightarrow\, R^1\Phi_*\ms End(\mc A)\, \longrightarrow\, 0
\end{equation*}
on $\Sigma$.

For every $i\,\geqslant  \, 1$, there is a natural isomorphism
\begin{equation*}
R^i\Phi_*\ms Hom(\mc A,\,\mc B)\,\,\stackrel{\sim}{\longrightarrow}\,\,
R^{i+1}\Phi_*\ms End(\mc A).
\end{equation*}
\end{theorem}

\begin{proof}
Application of $\ms Hom(\mc A,\, -)$ to the exact 
sequence in \eqref{universal-seq-C times Q}
produces an exact sequence
\begin{equation}\label{preliminary ses e1}
0\,\longrightarrow\, \ms End(\mc A)\,\longrightarrow\, \ms Hom(\mc A,\,p_1^*E)
\,\longrightarrow\, \ms Hom(\mc A,\,\mc B)\,\longrightarrow\, 0\,.
\end{equation} 
Using the projection formula and Proposition
\ref{propn-direct image of A} it follows that
	$$R^i\Phi_*\ms Hom(\mc A,\,p_1^*E)
\,=\,q_1^*E\otimes R^i\Phi_*(\mc A^{\vee})\,=\,0
$$
for all $i\,>\,0$. Therefore, applying $\Phi_*$ to \eqref{preliminary ses e1} produces
an exact sequence
\begin{equation}\label{u1}
0 \,\longrightarrow\, \Phi_*\ms End(\mc A)\,\longrightarrow\, \Phi_*\ms Hom(\mc A,\,p^*_1E)\,
\longrightarrow\,\Phi_*{\ms Hom}(\mc A,\,\mc B)\,\longrightarrow\,
R^1\Phi_*\ms End(\mc A)\,\longrightarrow 0\,\,.
\end{equation} 
Moreover, we get that there is an isomorphism
\begin{equation*}
R^i\Phi_*\ms Hom(\mc A,\,\mc B)\,\stackrel{\sim}{\longrightarrow}\, R^{i+1}\Phi_*\ms End(\mc A)
	\end{equation*}
	for every $i\,\geqslant  \, 1$. This proves the second part of the theorem.

To prove the first part of the theorem, it is enough to show that the image of the map 
	$$\Phi_*\ms Hom(\mc A,\,p^*_1E)\,\longrightarrow\, \Phi_*\ms Hom(\mc A,\, \mc B)$$ is isomorphic to 
	$ad(q_1^*E\big\vert_{\Sigma})$.
	
	Consider the commutative diagram
	\begin{equation*}
		\xymatrix{
			0\ar[r]& \ms Hom(p_1^*E, \,\mc A)\ar[r]\ar[d]& \ms Hom(p_1^*E,\, p_1^*E)\ar[r]\ar[d]& 
			\ms Hom(p_1^*E, \,\mc B)\ar[r]\ar[d]& 0\\
			0\ar[r]& \ms End(\mc A)\ar[r]& \ms Hom(\mc A,\,p_1^*E)\ar[r]& 
			\ms Hom(\mc A,\,\mc B)\ar[r]& 0\,.	
		}
\end{equation*}
Using projection formula together with Proposition \ref{propn-direct image of A}
it follows that
$$R^1\Phi_*\ms Hom(p_1^*E, \,\mc A)\,=\,q_1^*E^{\vee}\otimes R^1\Phi_*\mc A\,=\,0\,,$$
$$\Phi_*\ms Hom(\mc A,\,p_1^*E)
\,=\, q_1^*E\otimes \Phi_*(\mc A^{\vee})\,=\,\ms Hom(q_1^*E,\, q_1^*E)\,.$$
Consequently, applying $\Phi_*$ to the preceding diagram we get a diagram 
\begin{equation*}
\xymatrix{
0\ar[r]& \Phi_*\ms Hom(p_1^*E, \,\mc A)\ar[r]\ar[d]& \Phi_*\ms Hom(p_1^*E,\, p_1^*E)\ar[r]\ar[d]^{\cong} & 
\Phi_*\ms Hom(p_1^*E, \,\mc B)\ar[r]\ar[d] & 0\\
0\ar[r]& \Phi_*\ms End(\mc A)\ar[r]& \Phi_*\ms Hom(\mc A,\,p_1^*E)\ar[r]& 
\Phi_*\ms Hom(\mc A,\,\mc B). &	
}
\end{equation*} 
Thus, the image of the homomorphism
$\Phi_*\ms Hom(\mc A,\,p^*_1E)\,\longrightarrow\, \Phi_*{\ms Hom}(\mc A,\,\mc B)$
coincides with the image of the homomorphism
$\Phi_*\ms Hom(p_1^*E,\, \mc B)\,\longrightarrow\, \Phi_*{\ms Hom}(\mc A,\,\mc B)$.
Now consider the following commutative diagram in which the left vertical arrow 
is an isomorphism due to Lemma \ref{Hom(B,B)}
\begin{equation}\label{preliminary ses e2}
\begin{tikzcd}
0 \ar[r] & \mc O_{\mc Z} \ar[r] \ar[d,"\cong"] & {\ms End}\left(p_1^*E\big\vert_{\mc Z}\right) \ar[r] \ar[d] & 
{ad}\left(p_1^*E\big\vert_{\mc Z}\right)
\ar[r] \ar[d] & 0 \\
0 \ar[r] & {\ms Hom}(\mc B,\, \mc B) \ar[r] & {\ms Hom}(p_1^*E,\,\mc B)
\ar[r] & {\ms Hom}(\mc A,\, \mc B) &
\end{tikzcd}
\end{equation}
Applying $\Phi_*$ to it and using Corollary \ref{Phi_*B} we get the diagram
(the right vertical arrow is defined to be $\Xi$)
\begin{equation}\label{preliminary ses e3}
\xymatrix{
0 \ar[r] & \mc O_{\Sigma} \ar[r] \ar[d]^{\cong} & {\ms End}\left(q_1^*E\big\vert_{\Sigma}\right)
\ar[r] \ar[d]^{\cong} & {ad}\left(q_1^*E\big\vert_{\Sigma}\right)
\ar[r] \ar[d]^{\Xi} & 0 \\
0 \ar[r] & \Phi_*{\ms Hom}(\mc B, \,\mc B) \ar[r] & \Phi_*{\ms Hom}(p_1^*E,\,\mc B)
\ar[r] & \Phi_*{\ms Hom}(\mc A, \,\mc B) &
}
\end{equation}
Therefore, the image of the homomorphism
$\Phi_*{\ms Hom}(p_1^*E,\,\mc B)\,\longrightarrow\, \Phi_*{\ms Hom}(\mc A, \,\mc B)$ is same as the image of the map
${ad}(q_1^*E\big\vert_{\Sigma})\stackrel{\Xi}{\longrightarrow} \Phi_*{\ms Hom}(\mc A,
\,\mc B)$. But the above diagram shows that this map is an inclusion. Hence
its image is isomorphic to 
${ad}\left(q_1^*E\big\vert_{\Sigma}\right)$.
This completes the proof of the theorem.
\end{proof}

Note that it follows from Theorem \ref{preliminiary ses}
that the sheaf $R^i\Phi_*\ms End(\mc A)$ 
is supported on $\Sigma$ for every $i\,\geqslant  \, 1$.
By Corollary \ref{cor-direct image of structure sheaf},
the canonical map $R^1\Phi_*\ms End(\mc A)\,\longrightarrow\, R^1\Phi_*ad(\mc A)$
is an isomorphism.
Recall the relative adjoint Atiyah sequence (see \eqref{eqn rel ad-Atiyah}) for the 
locally free sheaf $\mc A$ on $C\times \mc Q$
\begin{equation}\label{rel Atiyah A on C times Q}
0\,\longrightarrow\, ad(\mc A)\,\longrightarrow\, at_C(\mc A)
\,\longrightarrow\, p_C^*T_C\,\longrightarrow\, 0\,.
\end{equation}
Applying $\Phi_*$ to \eqref{rel Atiyah A on C times Q} we get a map of sheaves
\begin{equation}\label{f18}
q_1^*T_C\,\longrightarrow\, R^1\Phi_*ad(\mc A)
\end{equation}
on $C\times C^{(d)}$.

For ease of notation, let
\begin{equation}\label{bar-q_1}
\overline{q_1}\,\, :\,\, \Sigma\,\longrightarrow\,C
\end{equation}
be the composite 
$\Sigma\,\hookrightarrow\, C\times C^{(d)}\,\stackrel{q_1}{\longrightarrow}\,C$,
where $q_1$ is the map in \eqref{t6}.

\begin{theorem}\label{description R1-Phi-End(A)}
The map in \eqref{f18} induces an isomorphism
$$ q_1^*T_C\big\vert_{\Sigma}\,=\,
\overline{q_1}^*T_C\,\stackrel{\sim}{\longrightarrow}\, R^1\Phi_*ad(\mc A)\, ,$$
where $\Phi$ is the map in \eqref{t7}.
Moreover, $R^i\Phi_*\ms End(\mc A)\,=\,0$ for all $i\,\geqslant  \, 2$.
\end{theorem}

\begin{proof}
We have already observed above that $R^1\Phi_*ad(\mc A)$
is supported on $\Sigma$. Thus, the map 
$q_1^*T_C\longrightarrow R^1\Phi_*ad(\mc A)$ 
factors through $q_1^*T_C\big\vert_{\Sigma}\longrightarrow\, R^1\Phi_*ad(\mc A)$. 

Let $(c,D)\in \Sigma$ be a point. Then
$c\,\in\, D$. We fix an ordering of the points of $D$ as mentioned at the
beginning of Section \ref{section cohomology of some sheaves}. 
We also choose this ordering in such a way that $c_d\,=\,c$. 
Associated to this ordering, we have the space $S_d$
as constructed in section \ref{section canonical bundle S_D}.
Recall, from \eqref{f7}, the map $g_d\,:\,S_d\,\longrightarrow\, \mc Q_D$.
We used the same notation to denote the composite map 
$S_d\,\longrightarrow \,\mc Q_D\,\longrightarrow \,\mc Q$.
Consider the composite
$$
T_{C,c}\,\,=\,\,\overline{q_1}^*T_C\big\vert_{(c,D)}\,\,\longrightarrow\,\,
 R^1\Phi_*ad(\mc A)\big\vert_{(c,D)}
$$
\begin{equation}\label{t1}
\longrightarrow\, 
H^1\left(\mc Q_D,\,ad\left(\mc A\big\vert_{c\times \mc Q_D}\right)\right) 
\,\stackrel{\sim}{\longrightarrow}\, H^1(S_d,\,ad(\mc A\big\vert_{c\times S_d})).
\end{equation}
The last map, which is induced by $g_d$, is an 
isomorphism by Corollary \ref{cor-cohomology over Q_D and S_D}.
Now, it can be proved that the composite in \eqref{t1} is an inclusion. This follows by identifying this map with
the infinitesimal deformation map associated to the family $A_d$ over $C\times S_d$, and then applying \cite[Proposition 4.4]{NR}. 

From Lemma \ref{lemma-cohomology computation end(G_d)} it follows
that $\dim H^1(S_d,\, \ms End(A_d\big\vert_{c\times S_d}))\,=\,1$.
Thus, $\dim H^1(S_d,\, ad(A_d\big\vert_{c\times S_d}))\,=\,1$.
In view of the injectivity of the composite in \eqref{t1}, this
implies that the composite map in \eqref{t1} is actually an isomorphism. Now by a base change argument we get the first statement of
the theorem. The second statement also follows from base change and Lemma \ref{lemma-cohomology computation end(G_d)}.
\end{proof}

\begin{corollary}\label{corollary main theorems}\mbox{}
\begin{enumerate}
\item There is the following short exact sequence on $\Sigma$
($\Xi$ is the map in \eqref{preliminary ses e3})
\begin{equation*}
0\,\longrightarrow\, ad(q_1^*E\big\vert_{\Sigma})\,\stackrel{\Xi}{\longrightarrow} \,
\Phi_*\ms Hom(\mc A,\,\mc B)\,\longrightarrow\, q_1^*T_C\big\vert_{\Sigma}\longrightarrow 0\,.
\end{equation*}

\item $R^i\Phi_*\ms Hom(\mc A,\,\mc B)\,=\,0$ for $i\,\geqslant  \, 1$.

\item $H^i(\Sigma,\, \Phi_*\ms Hom(\mc A,\,\mc B))\,\stackrel{\sim}{\longrightarrow}
\,H^i(\mc Z,\,\ms Hom(\mc A,\,\mc B))$ for all $i$.
	\end{enumerate}
\end{corollary}

\begin{proof}
Statement (1) follows by combining the short exact sequence in the 
statement of Theorem \ref{preliminiary ses} with the isomorphism in
Theorem \ref{description R1-Phi-End(A)}. Statement (2) follows using the isomorphism
in the statement of Theorem \ref{preliminiary ses} and 
the second assertion in Theorem \ref{description R1-Phi-End(A)}.
Statement (3) follows using the Leray spectral sequence and (2).
\end{proof}

\subsubsection{The tangent bundle of ${\mc Q}$}

The following proposition describing the tangent bundle of ${\mc Q}$ is standard.

\begin{proposition}\label{T_Q}
The tangent bundle of $\mc Q$ is $$T_{\mc Q}
\,\cong\, p_{2*}(\ms Hom(\mc A,\,\mc B)),$$
where $p_2$ is the projection in \eqref{t5}.
\end{proposition}

\begin{proof}
The proof is same as that of \cite[Theorem 7.1]{Str}.
\end{proof}

Recall that associated to a vector bundle $V$ on $C$ 
there is the Secant bundle on $C^{(d)}$
$${Sec}^d(V)\,:=\,q_{2*}[q_1^*V\big\vert_{\Sigma}]\,.$$

\begin{theorem}\label{main theorem}
The following statements hold:
\begin{enumerate}
\item There is a diagram
\begin{equation}\label{ses statement main theorem}
\begin{tikzcd}
0 \ar[r] & q_1^*ad(E)\big\vert_{\Sigma}\ar[r] \ar[d,equal] & q_1^*at(E)\big\vert_{\Sigma} \ar[r] 
\ar[d] & q_1^* T_C\big\vert_{\Sigma} \ar[r] \ar[d,"\cong"] & 0 \\
0 \ar[r] & q_1^*{ad}(E)\big\vert_{\Sigma} \ar[r,"\Xi"] & \Phi_*{\ms Hom}(\mc A,\,\mc B) \ar[r] & 
R^1\Phi_*ad(\mc A) \ar[r] & 0
\end{tikzcd}
\end{equation}
in which the squares commute up to a minus sign, where $q_1$ and $\Phi$ are
the maps in \eqref{t6} and \eqref{t7} respectively.
The right vertical arrow is the one coming from Theorem \ref{description R1-Phi-End(A)}.
In particular, the middle vertical arrow is an isomorphism.

\item $Sec^d(at(E))\,\,\stackrel{\sim}{\longrightarrow}\,\, \phi_*T_{\mc Q}$.

\item $R^i\phi_*T_{\mc Q}\,=\,0$ for all $i\,>\,0$.
\end{enumerate}
\end{theorem}

\begin{proof}
We will first construct diagram \eqref{ses statement main theorem}.
Denote by $\mc F$ the push-out of the diagram
	\begin{equation}\label{eqn-pushout}
	\begin{tikzcd}
	0 \ar[r] & ad(\mc A) \ar[r] \ar[d] & at_C(\mc A) \\
	& {\ms Hom}(\mc A,\,p_1^*E)/{\mc O}_{C\times {\mc Q}} .
	\end{tikzcd}
	\end{equation}
So using Snake lemma the following diagram is obtained:
	\begin{equation}\label{eqn-atiyah of A and E}
	\begin{tikzcd}
	& 0 \ar[d] & 0 \ar[d] & 
	& \\
	0 \ar[r] & ad(\mc A) \ar[r] \ar[d] & at_C(\mc A) \ar[d] \ar[r] &
	p_1^*T_C \ar[r] \ar[d, equal] & 0 \\
	0 \ar[r] & {\ms Hom}(\mc A,\,p_1^*E)/{\mc O}_{C\times{\mc Q}} \ar[r] \ar[d] & \mc F \ar[r] \ar[d] &
	p_1^* T_C \ar[r] & 0 \\
	& {\ms Hom}(\mc A,\,\mc B) \ar[r,equal] \ar[d] &
{\ms Hom}(\mc A,\,\mc B) \ar[d] &
& \\
& 0 & 0 & &
\end{tikzcd}
\end{equation}
Recall that by Corollary \ref{base change relative ad-atiyah} 
the relative adjoint Atiyah
sequence of $p_1^*E$ on $C\times\mc Q$, where $p_1$ is the projection
in \eqref{t5}, is simply 
the pullback of the relative adjoint Atiyah sequence for $E$. From
Corollary \ref{cor-compatibility of atiyah classes} it follows
that the middle row in \eqref{eqn-atiyah of A and E}
coincides with the push-out of relative adjoint Atiyah sequence of
$p_1^*E$ by the morphism $ad(p_1^*E)\,\longrightarrow\,
{\ms Hom}(A,\,p_1^*E)/\mc O$, that is, we have a diagram
\begin{equation}\label{eqn-atiyah of E}
\begin{tikzcd}
0 \ar[r] & p_1^*ad(E) \ar[r] \ar[d] & p_1^*at(E) \ar[r]
\ar[d] & p_1^* T_C \ar[r] \ar[d,equal] & 0 \\
0 \ar[r] & {\ms Hom}(\mc A,\,p_1^*E)/{\mc O}_{C\times{\mc Q}}\ar[r] & \mc F \ar[r]
& p_1^* T_C \ar[r] & 0 \\
\end{tikzcd}
\end{equation}
Combining (\ref{eqn-atiyah of A and E}) and (\ref{eqn-atiyah of E}) we get maps 
$$p_1^*at(E) \,\longrightarrow\, \mc F\,\longrightarrow\, {\ms Hom}(\mc A,\,
\mc B)\,.$$
Applying $\Phi_*$ we get a map
\begin{equation}\label{eqn-at(E) to Hom(A,B)}
q_1^*at(E)\,\,\longrightarrow\,\, \Phi_*{\ms Hom}(\mc A,\,\mc B)
\end{equation}
It can be shown that in the following diagram the squares commute up to a minus sign:
\begin{equation}\label{eqn-commutativity of at(E) and Hom(A,B)}
\begin{tikzcd}
0 \ar[r] & q_1^*ad(E) \ar[r] \ar[d] & q_1^*at(E) \ar[r]
\ar[d] & q_1^* T_C \ar[r] \ar[d] & 0 \\
0 \ar[r] & q_1^*{ad}(E)\big\vert_{\Sigma} \ar[r] & \Phi_*{\ms Hom}(\mc A,\,\mc B) \ar[r] & 
R^1\Phi_*ad(\mc A) \ar[r] & 0
\end{tikzcd}
\end{equation}
Here the bottom sequence is the one in 
Theorem \ref{preliminiary ses}. The middle vertical arrow is given by
\eqref{eqn-at(E) to Hom(A,B)} while the right vertical arrow is given by the 
boundary map of the sequence obtained by applying $\Phi_*$ to the 
relative adjoint Atiyah sequence of $\mc A$. 

The first assertion of the theorem now follows from Theorem \ref{description R1-Phi-End(A)} and from the fact that
all sheaves in the lower row of \eqref{ses statement main theorem} are supported on $\Sigma$.

It follows that we have an isomorphism 
$q^*_1 at(E)\big\vert_{\Sigma}\,\cong\, \Phi_*{\ms Hom}(\mc A,\,\mc B)$.
Applying $q_{2*}$ to this and using Proposition \ref{T_Q}
yields the second assertion.
The third assertion is deduced from 
Corollary \ref{corollary main theorems} using $q_2\big\vert_{\Sigma}$ 
and $p_2\big\vert_{\mc Z}$ are finite maps.
\end{proof}

\section*{Acknowledgements}

We thank the referee for helpful comments. 
The first author is partially supported by a J. C. Bose Fellowship (JBR/2023/000003).

\end{document}